\documentclass[12pt]{amsart}

\usepackage{amssymb,amsmath,graphicx,color,textcomp, amsthm,bbm,bbold, enumerate,booktabs}
\usepackage{chngcntr}
\usepackage{apptools}
\usepackage{mathrsfs}
\usepackage{tikz}
\usetikzlibrary{trees}

\AtAppendix{\counterwithin{theorem}{section}}
\definecolor{Red}{cmyk}{0,1,1,0}

\definecolor{verde}{cmyk}{1,0,1,0}

\definecolor{loka}{cmyk}{.5,0,1,.5}
\definecolor{azul}{cmyk}{1,1,0,0}

\numberwithin{equation}{section}

\newcommand{\be}{\begin{equation}}
\newcommand{\ee}{\end{equation}}

\newtheorem{theorem}{Theorem}

\newtheorem{definition}{Definition}

\begin{document}
\title{Mittag-Leffler functions and the truncated $\mathcal{V}$-fractional derivative}
\author{J. Vanterler da C. Sousa$^1$}
\address{$^1$ Department of Applied Mathematics, Institute of Mathematics,
 Statistics and Scientific Computation, University of Campinas --
UNICAMP, rua S\'ergio Buarque de Holanda 651,
13083--859, Campinas SP, Brazil\newline
e-mail: {\itshape \texttt{ra160908@ime.unicamp.br, capelas@ime.unicamp.br }}}

\author{E. Capelas de Oliveira$^1$}

\begin{abstract}We introduce a new derivative, the so-called truncated $\mathcal{V}$-fractional derivative for $\alpha$-differentiable functions through the six parameters truncated Mittag-Leffler function, which generalizes different fractional derivatives, recently introduced: conformable fractional derivatives, alternative fractional derivative, truncated alternative fractional derivative, $M$-fractional derivative and truncated $M$-fractional derivative.

This new truncated $\mathcal{V}$-fractional derivative satisfies properties of the entire order calculus, among them: linearity, product rule, quotient rule, function composition, and chain rule. Also, as in the case of the Caputo derivative, the derivative of a constant is zero. Since the six parameters Mittag-Leffler function is a generalization of Mittag-Leffler functions of one, two, three, four, and five parameters, we can extend some of the classic results of the entire order calculus, namely: Rolle's theorem, the mean value theorem and its extension. In addition, we present the theorem involving the law of exponents for derivatives and we calculated the truncated $\mathcal{V}$-fractional derivative of the two parameters Mittag-Leffler function.

Finally, we present the $\mathcal{V}$-fractional integral from which, as a natural consequence, new results appear as applications. Specifically, we generalize inverse property, the fundamental theorem of calculus, a theorem associated with classical integration by parts, and the mean value theorem for integrals. Also, we calculate the $\mathcal{V}$-fractional integral of the two parameters Mittag-Leffler function. Further, through the truncated $\mathcal{V}$-fractional derivative and the $\mathcal{V}$-fractional integral, we obtain a relation with the fractional derivative and integral in the Riemann-Liouville sense, in the case $0<\alpha<1$.
\end{abstract}
\maketitle

\section{Introduction}

The interest in the Mittag-Leffler functions increased due to their importance and applications in the no entire order calculus, the so-called fractional calculus, the fractional order differential equations and the Abel \cite{GOLL} type integral equations. In this sense, Caputo and Mainardi \cite{CAMI} have shown that Mittag-Leffler functions are present whenever fractional order derivatives are introduced into the constitutive equations associated with a body in the study of linear viscoelasticity. In 1980 \cite{RAYN}, Rabotnov also introduced, in linear viscoelasticity, Mittag-Leffler like functions; the connection for the use of fractional derivative operators in equations constituting linear viscoelasticity was established by Koeller \cite{KOR}.
	    
The simplest Mittag-Leffler function of type, $\mathbb{E}_{\alpha}(z)$ depends on a complex variable $z$ and a parameter $\alpha$ \cite{ML1}. As the exponential function is the solution of entire order differential equations with constant coefficients, the Mittag-Leffler function has an analogous role for solutions of no entire differential equations, and can be interpreted as a generalization of the exponential function.

Several generalizations of the Mittag-Leffler function have been proposed since 1903, when was introduced the Mittag-Leffler function \cite{ML1}. In 1905 Wiman \cite{ML2} generalized this function for two parameters, thus defining the two parameters Mittag-Leffler function. In 1971 Prabhakar \cite{ML3} introduced the so-called three parameters Mittag-Leffler function, a possible generalization of the two parameters Mittag-Leffler function. Shukla and Prajapati \cite{ML4} in 2007 introduced the four parameters Mittag-Leffler function. In 2012, Salim and Faraj \cite{ML6} defined a six parameters Mittag-Leffler function, this being a possible generalization of other Mittag-Leffler functions defined with less than six parameters. And recently, in 2013, Khan and Ahmed \cite{ML5} introduced the five parameters Mittag-Leffler function.

Several definitions of fractional derivatives of which we mention the fractional derivative: Riemann-Liouville, Caputo, Hadamard, Edérlyi-Kober, Katugampola, among others \cite{AHMJ,IP,UNT1}, have been introduced and consequently their applicability, of its importance for studies in the different areas of knowledge: Mathematics, Physics, Biology, Medicine and Engineering. On the other hand, accompanied by fractional derivatives, we have their respective fractional integrals \cite{AHMJ,IP,UNT}.

In 2014, Khalil \cite{KRHA}, introduced a new fractional integral and derivative that refers to the properties of the entire order calculus, called the conformable fractional derivative and the respective integral $\alpha$-fractional. Abdeljawad \cite{ABT}, presented a generalization for the conformable fractional derivative and for the $\alpha$-fractional derivative. In the same year, Katugampola \cite{UNT2} introduced the alternative fractional derivative and from the truncated exponential function, the truncated alternative fractional derivative, and consequently its respective integral $\alpha$-fractional. Recently, Sousa and Oliveira \cite{JEC1,JEC} introduced the $M$-fractional derivative and the truncated $M$-fractional derivative, which generalize the properties of the entire order calculus.

Motivated by the numerous applications involving the Mittag-Leffler functions in several areas, from the six parameters truncated Mittag-Leffler function and the gamma function, we introduced a new fractional derivative, the $\mathcal{V}$-fractional derivative, that generalizes the classical properties of the entire order calculus, as well as recent fractional derivatives $(\mbox{local derivative})$, of which we mention: the conformable fractional derivative, the alternative fractional derivative, the truncated alternative fractional derivative, the $M$-fractional derivative and the truncated $M$-fractional derivative, and $\mathcal{V}$-fractional derivative. We denote this new differential operator by $_{i}^{\rho }\mathcal{V}_{\gamma ,\beta ,\alpha }^{\delta ,p,q}$, where the parameter $\alpha$ , associated with the order of the derivative is such that $0<\alpha<1$, $\forall t>0$, where $\gamma ,\beta ,\rho ,\delta \in \mathbb{C}$ and $p,q>0$ such that ${Re}\left( \gamma \right) >0$, ${Re}\left( \beta \right) >0$, ${Re}\left( \rho \right) >0$, ${Re}\left( \delta \right) >0$, ${Re}\left( \gamma \right) +p\geq q$.

This article is organized as follows: in section 2 we present the Mittag-Leffler functions of one, two, three, four, five, and six parameters. In section 3, our main result, we present the concept of truncated $\mathcal{V}$-fractional derivative involving a six parameters truncated Mittag-Leffler function and the gamma function, as well as several theorems. In section 4, we present the respective fractional integral for which we demonstrate the inverse property. In section 5, we calculate the truncated $\mathcal{V}$-fractional derivative and integral of a two parameters Mittag-Leffler function. In section 6, we discuss two theorems that relate the truncated $\mathcal{V}$-fractional derivative and the $\mathcal{V}$-fractional integral with the derivative and fractional integral in the Riemann-Liouville sense. In section 7, we present the relation between the conformable fractional derivative, the alternative fractional derivative, the truncated alternative fractional derivative, the $M$-fractional derivative and the truncated $M$-fractional derivative, with the truncated $\mathcal{V}$-fractional derivative. Concluding remarks close the article.
\section{Prelimiaries}

As the exponential function is the solution of linear differential equations with constant coefficients, the Mittag-Leffler function is a solution of linear fractional differential equations with constant coefficients. Thus, the Mittag-Leffler function can be interpreted as a generalization of the exponential function. In this section we recover definitions of the Mittag-Leffler functions of: one, two, three, four, five and six parameters. We study the particular cases of the six parameters Mittag-Leffler function, important for the development of the article.

In 1903, Mittag-Leffler \cite{ML1} introduced the classic Mittag-Leffler function with only one
complex parameter.
\begin{definition}\label{def1}{\rm(One parameter Mittag-Leffler function)}. The Mittag-Leffler function is given by the series
\begin{equation}\label{A1}
\mathbb{E}_{\gamma }\left( z\right) =\overset{\infty }{\underset{k=0}{\sum }}\frac{z^{k}}{\Gamma \left( \gamma k+1\right) },
\end{equation}
with $\gamma \in \mathbb{C}$, $ {Re}\left( \gamma \right) >0$ and $\Gamma(z)$ is a gamma function, given by
\begin{equation}
\Gamma \left( z\right) =\int_{0}^{\infty }e^{-t}t^{z-1}dt, 
\end{equation}
$\mbox{Re}\left( z\right) >0$.
\end{definition}

In 1905, Wiman \cite{ML2} proposed and studied a generalization of the Mittag-Leffler function, the so-called two parameters Mittag-Leffler function.

\begin{definition}\label{def2}{\rm(Two parameters Mittag-Leffler function)}. The two parameters Mittag-Leffler function is given by the series
\begin{equation}\label{A2}
\mathbb{E}_{\gamma ,\beta }\left( z\right) =\overset{\infty }{\underset{k=0}{\sum }}%
\frac{z^{k}}{\Gamma \left( \gamma k+\beta \right) },
\end{equation}
with $\gamma ,\beta \in \mathbb{C}$, ${Re}\left( \gamma \right) >0$ and ${Re}\left( \beta
\right) >0$.
\end{definition}

In particular for $\beta=1$, we have $\mathbb{E}_{\gamma ,1}\left( z\right) =\mathbb{E}_{\gamma }\left( z\right) $, i.e; the Eq.(\ref{A1}). On the order hand, taking $\beta=\gamma=1$, we have $\mathbb{E}_{1,1}\left( z\right) =e^{z}$.

In 1971, Prabhakar \cite{ML3} introduced the three parameters Mittag-Leffler function.

\begin{definition}\label{def3}{\rm(Three parameters Mittag-Leffler function)}. Let $\gamma ,\beta ,\rho \in \mathbb{C}$ such that ${Re}\left( \gamma \right) >0$, ${Re}\left( \beta \right) >0$, $ {Re}\left( \rho \right) >0$, so
\begin{equation}\label{A3}
\mathbb{E}_{\gamma ,\beta }^{\rho }\left( z\right) =\overset{\infty }{\underset{k=0}{%
\sum }}\frac{\left( \rho \right) _{k}}{k!}\frac{z^{k}}{\Gamma \left( \gamma
k+\beta \right) },
\end{equation}
where $(\rho)_{k}$ is the Pochhammer symbol, defined by,
\begin{equation*}
\left( \rho \right) _{n}=\left\{ 
\begin{array}{ccc}
1, &  \text{for }n=0; \\ 
\alpha \left( \alpha +1\right)\cdots\left( \alpha +n-1\right), & \text{%
for }n\in \mathbb{N},
\end{array}
\right. 
\end{equation*}
which can be written in terms of the gamma function,
\begin{equation}\label{A4}
\left( \rho \right) _{k}=\rho \left( \rho +1\right) \cdot \cdot \cdot \left(
\rho +k-1\right) =\frac{\Gamma \left( \rho +k\right) }{\Gamma \left( \rho
\right) }.
\end{equation}
\end{definition}

In 2007, Shukla and Prajapati \cite{ML4}, introduced the four parameters Mittag-Leffler function, a generalization of Mittag-Leffler functions with one, two, and three parameters.

\begin{definition}\label{def4}{\rm(Four parameters Mittag-Leffler function)}. Let $\gamma ,\beta ,\rho \in \mathbb{C}$ and $q\in \left(0,1\right) \cup \mathbb{N}$ such that ${Re}\left( \gamma \right) >0$, ${Re}\left( \beta \right) >0$, ${Re}\left( \rho \right) >0$, so
\begin{equation}\label{A5}
\mathbb{E}_{\gamma ,\beta }^{\rho ,q}\left( z\right) =\overset{\infty }{\underset{k=0}%
{\sum }}\frac{\left( \rho \right) _{qk}}{k!}\frac{z^{k}}{\Gamma \left(
\gamma k+\beta \right) },
\end{equation}
where $(\rho)_{qk}$ is a generalization of the Pochhammer symbol, that is,
\begin{equation}\label{A6}
\left( \rho \right) _{qk}=\frac{\Gamma \left( \rho +qk\right) }{\Gamma\left( \rho \right) }.
\end{equation}
\end{definition}

Khan and Ahmed \cite{ML5}, introduced the five parameters Mittag-Leffler function.

\begin{definition}\label{def5}{\rm(Five parameters Mittag-Leffler function)}. Let $\gamma ,\beta ,\rho ,\delta \in \mathbb{C}$ and $q\in \left( 0,1\right) \cup \mathbb{N}$ such that ${Re}\left( \gamma \right) >0$, ${Re}\left( \beta\right) >0$, ${Re}\left( \rho \right) >0$, ${Re}\left(\delta \right) >0$, so
\begin{equation}\label{A7}
\mathbb{E}_{\gamma ,\beta ,\delta }^{\rho ,q}\left( z\right) =\overset{\infty }{%
\underset{k=0}{\sum }}\frac{\left( \rho \right) _{qk}}{\left( \delta \right)
_{k}}\frac{z^{k}}{\Gamma \left( \gamma k+\beta \right) },
\end{equation}
where $\left( \delta \right) _{k}$ and $\left( \rho \right) _{qk}$ are given by {\rm Eq.(\ref{A4})} and {\rm Eq.(\ref{A6})}.
\end{definition}

Finally, we present the six parameters Mittag-Leffler function as defined in 2012, by Salim and Faraj \cite{ML6}.

\begin{definition}\label{def6}{\rm(Six parameters Mittag-Leffler function)}. Let $\gamma ,\beta ,\rho ,\delta \in \mathbb{C}$ and $p,q>0$ such that ${Re}\left( \gamma \right) >0$, ${Re}\left( \beta \right) >0$, ${Re}\left( \rho \right) >0$, ${Re}\left( \delta \right) >0$, ${Re}\left( \gamma \right) +p\geq q$, so
\begin{equation}\label{A8}
\mathbb{E}_{\gamma ,\beta ,p}^{\rho ,\delta ,q}\left( z\right) =\overset{\infty }{%
\underset{k=0}{\sum }}\frac{\left( \rho \right) _{qk}}{\left( \delta \right)
_{pk}}\frac{z^{k}}{\Gamma \left( \gamma k+\beta \right) },
\end{equation}
where $\left( \delta \right) _{pk}$ and $\left( \rho \right) _{qk}$ are given by {\rm Eq.(\ref{A6})}.
\end{definition}

From Eq.(\ref{A8}), it is possible to obtain, as particular cases, the other Mittag-Leffler functions, namely:

\begin{enumerate}
\item For $p=1$, we get the five parameters Mittag-Leffler function, Eq.(\ref{A7}).

\item With $p=\delta=1$, we get the four parameters Mittag-Leffler function, Eq.(\ref{A5}).

\item In the case $p=\delta=q=1$, we get the three parameters Mittag-Leffler function, Eq.(\ref{A3}).

\item For $p=\delta=q=\rho=1$, we get the two parameters Mittag-Leffler function, Eq.(\ref{A2}).

\item With $p=\delta=q=\rho=\beta=1$, we obtain the one parameter Mittag-Leffler, Eq.(\ref{A1}).

\item Particularly, for $p=\delta=q=\rho=\beta=\gamma=1$, we recover the exponential function $e^{z}$.

\end{enumerate}

\section{Truncated $\mathcal{V}$-fractional derivative}

In this section, we define the truncated $\mathcal{V}$-fractional derivative using the six parameters truncated Mittag-Leffler function and the gamma function and we obtain several results similar to the results found in the classical calculus. From the definition, we present a theorem showing that the truncated $\mathcal{V}$-fractional derivative is linear, obeys the product rule and the composition of two $\alpha$-differentiable functions, the quotient rule and chain rule. It is also shown that the derivative of a constant is zero, as well as versions for Rolle's theorem, the mean value theorem, and an extension of the mean value theorem. The continuity of the truncated $\mathcal{V}$-fractional derivative is shown as in the entire order calculus, and in this sense, we present a theorem that refers to the law of exponents and the extension of the definition of the $n$ order truncated $\mathcal{V}$-fractional derivative.

Then, we begin with the definition of the six parameters truncated Mittag-Leffler function given by,
\begin{equation}\label{A9}
_{i}\mathbb{E}_{\gamma ,\beta ,p}^{\rho ,\delta ,q}\left( z\right) =\overset{i}{%
\underset{k=0}{\sum }}\frac{\left( \rho \right) _{qk}}{\left( \delta \right)
_{pk}}\frac{z^{k}}{\Gamma \left( \gamma k+\beta \right) },
\end{equation}
being $\gamma ,\beta ,\rho ,\delta \in \mathbb{C}$ and $p,q>0$ such that ${Re}\left( \gamma \right) >0$, ${Re}\left( \beta \right) >0$, ${Re}\left( \rho \right) >0$, ${Re}\left( \delta \right) >0$, ${Re}\left( \gamma \right) +p\geq q$ and $\left( \delta \right) _{pk}$, $\left( \rho \right) _{qk}$ given by Eq.(\ref{A6}).

From Eq.(\ref{A9}) and function $\Gamma(\beta)$, we introduce the following truncated function, denoted by $_{i}H^{\rho,\delta,q}_{\gamma,\beta,p}(z)$, by means of
\begin{equation}\label{A10}
_{i}H_{\gamma ,\beta ,p}^{\rho ,\delta ,q}\left( z\right) :=\Gamma \left( \beta
\right) \;_{i}\mathbb{E}_{\gamma ,\beta ,p}^{\rho ,\delta ,q}\left( z\right) =\Gamma
\left( \beta \right) \overset{i}{\underset{k=0}{\sum }}\frac{\left( \rho
\right) _{kq}}{\left( \delta \right) _{kp}}\frac{z^{k}}{\Gamma \left( \gamma
k+\beta \right) }.
\end{equation}

From Eq.(\ref{A10}), we define the truncated $\mathcal{V}$-fractional derivative that unifies other five fractional derivatives that refer to the classical properties of the integer order calculus.

In order to simplify notation, in this work, if the truncated $\mathcal{V}$-fractional derivative of order $\alpha$ according to Eq.(\ref{A11}) below of a function $f$ exists, we simply say that the $f$ function is $\alpha$-differentiable.

So, let's start with the following definition, which is a generalization of the usual definition of a derivative presented as a limit.

\begin{definition}\label{def7} Let $f:\left[ 0,\infty \right) \rightarrow \mathbb{R}$. For $0<\alpha <1$ the truncated $\mathcal{V}$-fractional derivative of $f$ of order $\alpha$, denoted by $_{i}^{\rho }\mathcal{V}_{\gamma ,\beta ,\alpha }^{\delta ,p,q}$, is defined by
\begin{equation}\label{A11}
_{i}^{\rho }\mathcal{V}_{\gamma ,\beta ,\alpha }^{\delta ,p,q}f\left( t\right) :=%
\underset{\epsilon \rightarrow 0}{\lim }\frac{f\left( t\;_{i}H_{\gamma ,\beta
,p}^{\rho ,\delta ,q}\left( \epsilon t^{-\alpha }\right) \right) -f\left(
t\right) }{\epsilon },
\end{equation}
for $\forall t>0$, $_{i}H_{\gamma ,\beta ,p}^{\rho ,\delta ,q}\left( \cdot\right) $ is a truncated function  as defined in {\rm Eq.(\ref{A10})} and being $\gamma ,\beta ,\rho ,\delta \in \mathbb{C}$ and $p,q>0$ such that ${Re}\left( \gamma \right) >0$, ${Re}\left( \beta \right) >0$, ${Re}\left( \rho \right) >0$, ${Re}\left( \delta \right) >0$, ${Re}\left( \gamma \right) +p\geq q$ and $\left( \delta \right) _{pk}$, $\left( \rho \right) _{qk}$ given by {\rm Eq.(\ref{A6})}. 
\end{definition}

Note that, if $f$ is differentiable in some $(0,a)$, $a>0$ and $\underset{t\rightarrow 0^{+}}{\lim }\left( _{i}^{\rho }\mathcal{V}_{\gamma ,\beta ,\alpha }^{\delta ,p,q}f\left( t\right) \right) $ exist, then we have
\begin{equation*}
_{i}^{\rho }\mathcal{V}_{\gamma ,\beta ,\alpha }^{\delta ,p,q}f\left(
0\right) =\underset{t\rightarrow 0^{+}}{\lim }\left( _{i}^{\rho }\mathcal{V}_{\gamma ,\beta ,\alpha }^{\delta ,p,q}f\left( t\right) \right).
\end{equation*}

\begin{theorem}\label{teo1} If the function $f:\left[ 0,\infty \right) \rightarrow \mathbb{R}$ is $\alpha $-differentiable for $t_{0}>0$, with $0<\alpha \leq 1$, then $f$ is continuous in $t_{0}$. 
\end{theorem} 
\begin{proof} In fact, consider the following identity
\begin{equation} \label{A12}
f\left( t_{0}\Gamma \left( \beta \right) \;_{i}\mathbb{E}_{\gamma ,\beta ,p}^{\rho ,\delta ,q}\left( \varepsilon t_{0}^{-\alpha }\right) \right) -f\left( t_{0}\right) =\left( \frac{f\left( t_{0}\Gamma \left( \beta \right)\;_{i}\mathbb{E}_{\gamma ,\beta ,p}^{\rho ,\delta ,q}\left( \varepsilon t_{0}^{-\alpha }\right) \right) -f\left( t_{0}\right) }{\varepsilon }\right) \varepsilon . 
\end{equation}

Taking the limit $\varepsilon \rightarrow 0$ on both sides of {\rm{Eq}.(\ref{A12})}, we have
\begin{eqnarray*}
\underset{\varepsilon \rightarrow 0}{\lim }f\left( t_{0}\Gamma \left( \beta \right)\;_{i}\mathbb{E}_{\gamma ,\beta ,p}^{\rho ,\delta ,q}\left( \varepsilon t_{0}^{-\alpha}\right) \right) -f\left( t_{0}\right)  &=&\underset{\varepsilon \rightarrow0}{\lim }\left( \frac{f\left( t_{0}\Gamma \left( \beta \right)\;_{i}\mathbb{E}_{\gamma ,\beta ,p}^{\rho
,\delta ,q}\left( \varepsilon t_{0}^{-\alpha }\right) \right) -f\left(
t_{0}\right) }{\varepsilon }\right) \underset{\varepsilon \rightarrow 0}{\lim }\varepsilon  \\
&=&_{i}^{\rho }\mathcal{V}_{\gamma ,\beta ,\alpha }^{\delta ,p,q}f\left(t\right) \underset{\varepsilon \rightarrow 0}{\lim }\varepsilon  \\&=&0.
\end{eqnarray*}

Then, $f$ is continuous in $t_{0}$.
\end{proof}

Introducing the series representation for the truncated function $_{i}H_{\gamma ,\beta,p}^{\rho ,\delta ,q}\left( \cdot \right) $, we have
\begin{equation}\label{A13}
f\left( t\;\Gamma \left( \beta \right) _{i}\mathbb{E}_{\gamma ,\beta
,p}^{\rho ,\delta ,q}\left( \varepsilon t^{-\alpha }\right) \right) =f\left(t\Gamma \left( \beta \right) \overset{i}{\underset{k=0}{\sum }}\frac{\left(\rho \right) _{kq}}{\left( \delta \right) _{kp}}\frac{\left( \varepsilon
t^{-\alpha }\right) ^{k}}{\Gamma \left( \gamma k+\beta \right) }\right) ,
\end{equation}
then, taking the limit $\varepsilon \rightarrow 0$ on both sides of Eq.(\ref{A13}) and as $f$ is continuous, we get
\begin{eqnarray*}
\underset{\varepsilon \rightarrow 0}{\lim }f\left( t\;\Gamma \left( \beta\right) _{i}\mathbb{E}_{\gamma ,\beta ,p}^{\rho ,\delta ,q}\left(
\varepsilon t^{-\alpha }\right) \right)  &=&\underset{\varepsilon
\rightarrow 0}{\lim }f\left( t\Gamma \left( \beta \right) \overset{i}{
\underset{k=0}{\sum }}\frac{\left( \rho \right) _{kq}}{\left( \delta \right) _{kp}}\frac{\left( \varepsilon t^{-\alpha }\right) ^{k}}{\Gamma \left(\gamma k+\beta \right) }\right)   \notag \\
&=&f\left( t\Gamma \left( \beta \right) \underset{\varepsilon \rightarrow 0}{\lim }\overset{i}{\underset{k=0}{\sum }}\frac{\left( \rho \right) _{kq}}{\left( \delta \right) _{kp}}\frac{\left( \varepsilon t^{-\alpha }\right) ^{k}}{\Gamma \left( \gamma k+\beta \right) }\right) .
\end{eqnarray*}

Besides that, we have
\begin{eqnarray}\label{A14}
_{i}\mathbb{E}_{\gamma ,\beta ,p}^{\rho ,\delta ,q}\left( \varepsilon
t^{-\alpha }\right)  &=&\overset{i}{\underset{k=0}{\sum }}\frac{\left( \rho
\right) _{kq}}{\left( \delta \right) _{kp}}\frac{\left( \varepsilon
t^{-\alpha }\right) ^{k}}{\Gamma \left( \gamma k+\beta \right) }  \notag \\
&=&\frac{1}{\Gamma \left( \beta \right) }+\frac{\Gamma \left( \rho +q\right)
\Gamma \left( \delta \right) \varepsilon t^{1-\alpha }}{\Gamma \left( \rho
\right) \Gamma \left( \delta +p\right) \Gamma \left( \gamma +\beta \right) }+%
\frac{\Gamma \left( \rho +2q\right) \Gamma \left( \delta \right) \left(
\varepsilon t^{-\alpha }\right) ^{2}}{\Gamma \left( \rho \right) \Gamma
\left( \delta +2p\right) \Gamma \left( 2\gamma +\beta \right) }  \notag \\
&&+\frac{\Gamma \left( \rho +3q\right) \Gamma \left( \delta \right) \left(
\varepsilon t^{-\alpha }\right) ^{3}}{\Gamma \left( \rho \right) \Gamma
\left( \delta +3p\right) \Gamma \left( 3\gamma +\beta \right) }+\cdot \cdot
\cdot +\frac{\Gamma \left( \rho +iq\right) \Gamma \left( \delta \right)
\left( \varepsilon t^{-\alpha }\right) ^{i}}{\Gamma \left( \rho \right)
\Gamma \left( \delta +ip\right) \Gamma \left( i\gamma +\beta \right) }.\notag \\
\end{eqnarray}

Taking the limit $\varepsilon \rightarrow 0$ on both sides of {\rm Eq.(\ref{A14})}, we have
\begin{equation*}
\underset{\varepsilon \rightarrow 0}{\lim }\overset{i}{\underset{k=0}{\sum }}%
\frac{\left( \rho \right) _{kq}}{\left( \delta \right) _{kp}}\frac{\left(
\varepsilon t^{-\alpha }\right) ^{k}}{\Gamma \left( \gamma k+\beta \right) }=%
\frac{1}{\Gamma \left( \beta \right) },
\end{equation*}
then, we conclude that
\begin{equation*}
\underset{\varepsilon \rightarrow 0}{\lim }f\left( t\;\Gamma \left( \beta
\right) _{i}\mathbb{E}_{\gamma ,\beta ,p}^{\rho ,\delta ,q}\left(
\varepsilon t^{-\alpha }\right) \right) =f\left( t\right) .
\end{equation*}

Here, we present the theorem that includes the main classical properties of entire order calculus. For the chain rule, will be check by an example, as we will see below. Let's do here, just the demonstration of the chain rule. For other items, the reasoning is the same as used in Theorem \ref{teo2} found in Sousa and Oliveira \cite{JEC1}.

\begin{theorem}\label{teo2} Let $0<\alpha\leq 1$, $a,b\in\mathbb{R}$, $\gamma ,\beta ,\rho ,\delta \in \mathbb{C}$ and $p,q>0$ such that ${Re}\left( \gamma \right) >0$, ${Re}\left( \beta \right) >0$, ${Re}\left( \rho \right) >0$, ${Re}\left( \delta \right) >0$, ${Re}\left( \gamma \right) +p\geq q$ and $f,g$ $\alpha$-differentiable, for $t>0$. Then,
\begin{enumerate}
\item $_{i}^{\rho }\mathcal{V}_{\gamma ,\beta ,\alpha }^{\delta ,p,q}\left( af+bg\right)
\left( t\right) =a\left( _{i}^{\rho }\mathcal{V}_{\gamma ,\beta ,\alpha }^{\delta
,p,q}f\left( t\right) \right) +b\left( _{i}^{\rho }\mathcal{V}_{\gamma ,\beta ,\alpha
}^{\delta ,p,q}g\left( t\right) \right) $

\item $_{i}^{\rho }\mathcal{V}_{\gamma ,\beta ,\alpha }^{\delta ,p,q}\left( f\cdot g\right)
\left( t\right) =f\left( t\right) \left( _{i}^{\rho }\mathcal{V}_{\gamma ,\beta
,\alpha }^{\delta ,p,q}g\left( t\right) \right) +g\left( t\right) \left(
_{i}^{\rho }\mathcal{V}_{\gamma ,\beta ,\alpha }^{\delta ,p,q}f\left( t\right) \right) 
$

\item $_{i}^{\rho }\mathcal{V}_{\gamma ,\beta ,\alpha }^{\delta ,p,q}\left( \frac{f}{g}%
\right) \left( t\right) =\displaystyle\frac{g\left( t\right) \left( _{i}^{\rho }\mathcal{V}_{\gamma
,\beta ,\alpha }^{\delta ,p,q}f\left( t\right) \right) -f\left( t\right)
\left( _{i}^{\rho }\mathcal{V}_{\gamma ,\beta ,\alpha }^{\delta ,p,q}g\left( t\right)
\right) }{\left[ g\left( t\right) \right] ^{2}}$

\item $_{i}^{\rho }\mathcal{V}_{\gamma ,\beta ,\alpha }^{\delta ,p,q}\left( c\right) =0$, where $f(t)=c$ is a constant.

\item {\rm\text{(Chain rule)}} If $f$ is differentiable, then $_{i}^{\rho }\mathcal{V}_{\gamma ,\beta ,\alpha }^{\delta ,p,q}f\left( t\right) =\displaystyle\frac{t^{1-\alpha }\Gamma \left( \beta \right) \left( \rho \right) _{q}}{\Gamma\left( \gamma +\beta \right) \left( \delta \right) _{p}}\frac{df\left(
t\right) }{dt},$ being $(\rho)_{q}$ and $(\delta)_{p}$ the symbol of Pochhammer, given by {\rm Eq.(\ref{A4})}.

\begin{proof} From {\rm\text{Eq}.\rm(\ref{A14})}, we have
\begin{equation*}
t\;\Gamma \left( \beta \right) _{i}\mathbb{E}_{\gamma ,\beta ,p}^{\rho
,\delta ,q}\left( \varepsilon t^{-\alpha }\right) =t+\frac{\Gamma \left(
\beta \right) }{\Gamma \left( \gamma +\beta \right) }\frac{\left( \rho
\right) _{q}}{\left( \delta \right) _{p}}\varepsilon t^{1-\alpha }+O\left(
\varepsilon ^{2}\right) ,
\end{equation*}
with $\left( \rho \right) _{q}=\displaystyle\frac{\Gamma \left( \rho +q\right) }{\Gamma \left( \rho \right) }$ and $\displaystyle\left( \delta \right) _{p}=\frac{\Gamma \left( \delta +p\right) }{\Gamma \left( \delta \right) }$.

Introducing the following change,
\begin{equation*}
h=\varepsilon t^{1-\alpha }\left( \frac{\Gamma \left( \beta \right) \left(
\rho \right) _{q}}{\Gamma \left( \gamma +\beta \right) \left( \delta \right)
_{p}}+O\left( \varepsilon \right) \right) \Rightarrow \varepsilon =\frac{h}{%
t^{1-\alpha }\left( \displaystyle\frac{\Gamma \left( \beta \right) \left( \rho \right)
_{q}}{\Gamma \left( \gamma +\beta \right) \left( \delta \right) _{p}}%
+O\left( \varepsilon \right) \right) },
\end{equation*}%
we conclude that 
\begin{eqnarray*}
_{i}^{\rho }\mathcal{V}_{\gamma ,\beta ,\alpha }^{\delta ,p,q}f\left( t\right)  &=&%
\underset{\varepsilon \rightarrow 0}{\lim }\frac{\displaystyle\frac{f\left(
t+h\right) -f\left( t\right) }{ht^{\alpha -1}}}{\displaystyle\frac{\Gamma \left( \beta
\right) \left( \rho \right) _{q}}{\Gamma \left( \gamma +\beta \right) \left(
\delta \right) _{p}}\left( 1+\displaystyle\frac{\Gamma \left( \gamma +\beta \right)
\left( \delta \right) _{p}}{\Gamma \left( \beta \right) \left( \rho \right)
_{q}}+O\left( \varepsilon \right) \right) } \\
&=&\frac{t^{1-\alpha }}{\displaystyle\frac{\Gamma \left( \beta \right) \left( \rho
\right) _{q}}{\Gamma \left( \gamma +\beta \right) \left( \delta \right) _{p}}%
}\underset{\varepsilon \rightarrow 0}{\lim }\frac{\displaystyle\frac{f\left(
t+h\right) -f\left( t\right) }{h}}{1+\displaystyle\frac{\Gamma \left( \gamma +\beta
\right) \left( \delta \right) _{p}}{\Gamma \left( \beta \right) \left( \rho
\right) _{q}}+O\left( \varepsilon \right) } \\
&=&\frac{t^{1-\alpha }\Gamma \left( \beta \right) \left( \rho \right) _{q}}{%
\Gamma \left( \gamma +\beta \right) \left( \delta \right) _{p}}\frac{%
df\left( t\right) }{dt},
\end{eqnarray*}
with $t>0$.
\end{proof}

\item $_{i}^{\rho }\mathcal{V}_{\gamma ,\beta ,\alpha }^{\delta ,p,q}\left( f\circ
g\right) \left( t\right) =f^{\prime }\left( g\left( t\right) \right) \left(_{i}^{\rho }\mathcal{V}_{\gamma ,\beta ,\alpha }^{\delta ,p,q}g\left(t\right) \right) $, for $f$ differentiable in $g(t)$.
\end{enumerate}
\end{theorem}

Considering $p=q=\delta=\rho=\beta=1$, in item 5 of the Theorem \ref{teo2}, we have
$_{i}^{1}\mathcal{V}_{\gamma ,1,\alpha }^{1,1,1}f\left( t\right) =\displaystyle\frac{t^{1-\alpha }}{\Gamma \left( \gamma +1\right) }\frac{df\left( t\right) }{dt}$, i.e; the chain rule for $M$-fractional derivative and truncated $M$-fractional derivative. For $\gamma=1$, we have $_{i}^{1}V_{1,1,\alpha }^{1,1,1}f\left( t\right) =t^{1-\alpha }\displaystyle\frac{df\left( t\right) }{dt}$, that is, the chain rule for the conformable fractional derivative, the alternative fractional derivative, and the truncated alternative fractional derivative.

Let us omit the proof of the Theorem \ref{teo3} and Theorem \ref{teo4}. The proof follows the chain rule, as seen in item 5 of the Theorem \ref{teo2}.

\begin{theorem}\label{teo3} Let $0<\alpha\leq 1$, $\gamma ,\beta ,\rho ,\delta \in \mathbb{C}$ and $p,q>0$ such that ${Re}\left( \gamma \right) >0$, ${Re}\left( \beta \right) >0$, ${Re}\left( \rho \right) >0$, ${Re}\left( \delta \right) >0$, ${Re}\left( \gamma \right) +p\geq q$ and $f,g$ $\alpha$-differentiable, for $t>0$. Then, we have the following results
\begin{enumerate}

\item $_{i}^{\rho }\mathcal{V}_{\gamma ,\beta ,\alpha }^{\delta ,p,q}\left(
e^{at}\right) =\displaystyle\frac{\Gamma \left( \beta \right) \left( \rho \right) _{q}}{%
\Gamma \left( \gamma +\beta \right) \left( \delta \right) _{p}}t^{1-\alpha
}a e^{at}$

\item $_{i}^{\rho }\mathcal{V}_{\gamma ,\beta ,\alpha }^{\delta ,p,q}\sin \left(
at\right) =\displaystyle\frac{\Gamma \left( \beta \right) \left( \rho \right) _{q}}{%
\Gamma \left( \gamma +\beta \right) \left( \delta \right) _{p}}t^{1-\alpha
}a\cos \left( at\right) $

\item $_{i}^{\rho }\mathcal{V}_{\gamma ,\beta ,\alpha }^{\delta ,p,q}\cos \left(
at\right) =-\displaystyle\frac{\Gamma \left( \beta \right) \left( \rho \right) _{q}}{%
\Gamma \left( \gamma +\beta \right) \left( \delta \right) _{p}}t^{1-\alpha
}a\sin \left( at\right) $

\item $_{i}^{\rho }\mathcal{V}_{\gamma ,\beta ,\alpha }^{\delta ,p,q}\left(
t^{a}\right) =\displaystyle\frac{\Gamma \left( \beta \right) \left( \rho \right) _{q}}{%
\Gamma \left( \gamma +\beta \right) \left( \delta \right) _{p}}at^{1-\alpha }$

\item $_{i}^{\rho }\mathcal{V}_{\gamma ,\beta ,\alpha }^{\delta ,p,q}\left( \frac{%
t^{\alpha }}{\alpha }\right) =\displaystyle\frac{\Gamma \left( \beta \right) \left( \rho
\right) _{q}}{\Gamma \left( \gamma +\beta \right) \left( \delta \right) _{p}}$
\end{enumerate}
\end{theorem}

\begin{theorem}\label{teo4} Let $0<\alpha\leq 1$, $\gamma ,\beta ,\rho ,\delta \in \mathbb{C}$ and $p,q>0$ such that ${Re}\left( \gamma \right) >0$, ${Re}\left( \beta \right) >0$, ${Re}\left( \rho \right) >0$, ${Re}\left( \delta \right) >0$, ${Re}\left( \gamma \right) +p\geq q$ and $f,g$ $\alpha$-differentiable, for $t>0$. Then, we have the following results 
\begin{enumerate}
\item $_{i}^{\rho }\mathcal{V}_{\gamma ,\beta ,\alpha }^{\delta ,p,q}\sin
\left( \frac{t^{\alpha }}{\alpha }\right) =\displaystyle\frac{\Gamma \left( \beta \right)
\left( \rho \right) _{q}}{\Gamma \left( \gamma +\beta \right) \left( \delta
\right) _{p}}\cos \left( \frac{t^{\alpha }}{\alpha }\right) $

\item $_{i}^{\rho }\mathcal{V}_{\gamma ,\beta ,\alpha }^{\delta ,p,q}\cos
\left( \frac{t^{\alpha }}{\alpha }\right) =-\displaystyle\frac{\Gamma \left( \beta
\right) \left( \rho \right) _{q}}{\Gamma \left( \gamma +\beta \right) \left(
\delta \right) _{p}}\sin \left( \frac{t^{\alpha }}{\alpha }\right) $

\item $_{i}^{\rho }\mathcal{V}_{\gamma ,\beta ,\alpha }^{\delta ,p,q}\left( e^{%
\frac{t^{\alpha }}{\alpha }}\right) =\displaystyle\frac{\Gamma \left( \beta \right)
\left( \rho \right) _{q}}{\Gamma \left( \gamma +\beta \right) \left( \delta
\right) _{p}}e^{\frac{t^{\alpha }}{\alpha }}$
\end{enumerate}
\end{theorem}

Theorem \ref{der} bellow proves that commutativity property depends on the type of fractional operator.

\begin{theorem}\label{der} Let $_{i}^{\rho }\mathcal{V}_{\gamma ,\beta ,\alpha }^{\delta ,p,q}f(t)$ and $_{i}^{\rho }\mathcal{V}_{\gamma ,\beta ,\mu }^{\delta ,p,q}f(t)$ truncated $\mathcal{V}$-fractional derivative of the order $\alpha$ {\rm($0<\alpha<1$)} and {\rm$\mu$ ($0<\mu<1$)}, respectively. So, we have
\begin{equation}
_{i}^{\rho }\mathcal{V}_{\gamma ,\beta ,\alpha }^{\delta ,p,q}\left( _{i}^{\rho
}\mathcal{V}_{\gamma ,\beta ,\mu }^{\delta ,p,q}f\left( t\right) \right) =\frac{\Gamma
\left( \beta \right) \left( \rho \right) _{q}}{\Gamma \left( \gamma +\beta
\right) \left( \delta \right) _{p}}\left( \left( 1-\mu \right) \left(
_{i}^{\rho }\mathcal{V}_{\gamma ,\beta ,\alpha +\mu }^{\delta ,p,q}f\left( t\right)
\right) +t\left( _{i}^{\rho }\mathcal{V}_{\gamma ,\beta ,\alpha +\mu }^{\delta
,p,q}f^{\prime }\left( t\right) \right) \right) .
\end{equation}
\end{theorem}

\begin{proof}

In fact using the chain rule in item 5 of {\rm Theorem \ref{teo2}}, we have
\begin{eqnarray}\label{ZE1}
_{i}^{\rho }\mathcal{V}_{\gamma ,\beta ,\alpha }^{\delta ,p,q}\left( _{i}^{\rho
}\mathcal{V}_{\gamma ,\beta ,\mu }^{\delta ,p,q}f\left( t\right) \right)  &=&\;_{i}^{\rho }\mathcal{V}_{\gamma ,\beta ,\alpha }^{\delta ,p,q} \left( \frac{%
\Gamma \left( \beta \right) \left( \rho \right) _{q}}{\Gamma \left( \gamma
+\beta \right) \left( \delta \right) _{p}}t^{1-\mu }f^{\prime }\left(
t\right) \right)   \notag \\
&=&t^{1-\alpha }\left( \frac{\Gamma \left( \beta \right) \left( \rho \right)
_{q}}{\Gamma \left( \gamma +\beta \right) \left( \delta \right) _{p}}\right)
^{2}\frac{d}{dt}\left( t^{1-\mu }f^{\prime }\left( t\right) \right)   \notag
\\
&=&\left( \frac{\Gamma \left( \beta \right) \left( \rho \right) _{q}}{\Gamma
\left( \gamma +\beta \right) \left( \delta \right) _{p}}\right) ^{2}\left(
\left( 1-\mu \right) t^{1-\alpha -\mu }f^{\prime }\left( t\right)
+t^{2-\alpha -\mu }f^{\prime \prime }\left( t\right) \right).\notag
\\
\end{eqnarray}

By  \rm {Definition \ref{def1}}, we have
\begin{equation}\label{ZE2}
_{i}^{\rho }\mathcal{V}_{\gamma ,\beta ,\alpha +\mu }^{\delta ,p,q}f\left( t\right) =%
\frac{\Gamma \left( \beta \right) \left( \rho \right) _{q}}{\Gamma \left(
\gamma +\beta \right) \left( \delta \right) _{p}}\left( t^{1-\alpha -\mu
}f^{\prime }\left( t\right) \right).
\end{equation}

So, replacing {\rm Eq.(\ref{ZE2})} in {\rm Eq.(\ref{ZE1})}, we conclude that
\begin{eqnarray*}
_{i}^{\rho }\mathcal{V}_{\gamma ,\beta ,\alpha }^{\delta ,p,q}\left( _{i}^{\rho
}\mathcal{V}_{\gamma ,\beta ,\mu }^{\delta ,p,q}f\left( t\right) \right)  &=&\frac{%
\Gamma \left( \beta \right) \left( \rho \right) _{q}}{\Gamma \left( \gamma
+\beta \right) \left( \delta \right) _{p}}\left( 
\begin{array}{c}
\left( \displaystyle\frac{\Gamma \left( \beta \right) \left( \rho \right) _{q}}{\Gamma
\left( \gamma +\beta \right) \left( \delta \right) _{p}}\right) \left( 1-\mu
\right) t^{1-\alpha -\mu }f^{\prime }\left( t\right)  \\ 
+\left( \displaystyle\frac{\Gamma \left( \beta \right) \left( \rho \right) _{q}}{\Gamma
\left( \gamma +\beta \right) \left( \delta \right) _{p}}\right) t^{2-\alpha
-\mu }f^{\prime \prime }\left( t\right) 
\end{array}\right)   \notag \\
&=&\frac{\Gamma \left( \beta \right) \left( \rho \right) _{q}}{\Gamma \left(
\gamma +\beta \right) \left( \delta \right) _{p}}\left( \left( 1-\mu \right)
\left( _{i}^{\rho }\mathcal{V}_{\gamma ,\beta ,\alpha +\mu }^{\delta ,p,q}f\left(
t\right) \right) +t\left( _{i}^{\rho }\mathcal{V}_{\gamma ,\beta ,\alpha +\mu
}^{\delta ,p,q}f^{\prime }\left( t\right) \right) \right). 
\end{eqnarray*}
\end{proof}

From Theorem \ref{der}, follow $_{i}^{\rho }\mathcal{V}_{\gamma ,\beta ,\alpha }^{\delta ,p,q}\left( _{i}^{\rho }\mathcal{V}_{\gamma ,\beta ,\mu }^{\delta ,p,q}f\left( t\right) \right) \neq \left(_{i}^{\rho }\mathcal{V}_{\gamma ,\beta ,\alpha +\mu }^{\delta ,p,q}f\left( t\right)
\right)$.

The proof of Rolle's theorem, next, will be omitted,  because it is analogous, as the $M$-fractional derivative \cite{JEC1,JEC}.

\begin{theorem} {\rm\text{(Rolle's theorem for fractional $\alpha$-differentiable functions)}}
Let $a>0$, and $f:\left[ a,b\right] \rightarrow \mathbb{R}$ a function with the properties:
\begin{enumerate}
\item $f$ is continuous in $[a,b]$.
\item $f$ is $\alpha$-differentiable in $(a,b)$ for some $\alpha\in(0,1)$.
\item  $f(a)=f(b)$.
\end{enumerate}
Then, $\exists c\in(a,b)$, such that $_{i}^{\rho }\mathcal{V}_{\gamma ,\beta ,\alpha }^{\delta ,p,q}f(c)=0$, with $\gamma ,\beta ,\rho ,\delta \in \mathbb{C}$ and $p,q>0$ such that ${Re}\left( \gamma \right) >0$, ${Re}\left( \beta \right) >0$, ${Re}\left( \rho \right) >0$, ${Re}\left( \delta \right) >0$ and ${Re}\left( \gamma \right) +p\geq q$.
\end{theorem}

\begin{theorem} {\rm\text{(Mean-value theorem for fractional $\alpha$-differentiable functions)}}
Let $a>0$ and $f:\left[ a,b\right] \rightarrow \mathbb{R}$ a function with the properties:
\begin{enumerate}
\item $f$ is continuous in $[a,b]$.
\item $f$ is $\alpha$-differentiable in $(a,b)$ for some $\alpha\in(0,1)$.
\end{enumerate}
Then, $\exists c\in(a,b)$, such that
\begin{equation*}
_{i}^{\rho }\mathcal{V}_{\gamma ,\beta ,\alpha }^{\delta ,p,q}f(c) =\frac{f\left( b\right) -f\left(a\right) }{\displaystyle\frac{b^{\alpha }}{\alpha }-\frac{a^{\alpha }}{\alpha }},
\end{equation*}
with $\gamma ,\beta ,\rho ,\delta \in \mathbb{C}$ and $p,q>0$ such that ${Re}\left( \gamma \right) >0$, ${Re}\left( \beta \right) >0$, ${Re}\left( \rho \right) >0$, ${Re}\left( \delta \right) >0$ and ${Re}\left( \gamma \right) +p\geq q$.
\end{theorem}

\begin{proof}
Considered the function
\begin{equation}
g\left( x\right) =f\left( x\right) -f\left( a\right) -\frac{{\Gamma \left(
\gamma +\beta \right) }\left( \delta \right) _{p}}{{\Gamma \left( \beta
\right) }\left( \rho \right) _{q}}\left( \displaystyle\frac{f\left( b\right)
-f\left( a\right) }{\displaystyle\frac{1}{\alpha }b^{\alpha }-\displaystyle%
\frac{1}{\alpha }a^{\alpha }}\right) \left( \frac{1}{\alpha }x^{\alpha }-%
\frac{1}{\alpha }a^{\alpha }\right) .  \label{E}
\end{equation}

The function $g$ satisfies the conditions of Rolle's theorem. Then there is $c\in(a,b)$ such that $_{i}^{\rho }V_{\gamma ,\beta ,\alpha }^{\delta ,p,q}f\left( c\right) =0$. Taking the truncated $\mathcal{V}$-fractional derivative $_{i}^{\rho }V_{\gamma ,\beta ,\alpha }^{\delta ,p,q}\left( \cdot \right) $ on both sides of {\rm{Eq}.(\ref{E})} and the relation $_{i}^{\rho }V_{\gamma ,\beta ,\alpha }^{\delta ,p,q}\left( \frac{t^{\alpha }}{\alpha }\right) =\displaystyle\frac{{\Gamma \left( \beta \right) }\left( \rho \right)_{q}}{{\Gamma \left( \gamma +\beta \right) }\left( \delta \right) _{p}}$ and $%
_{i}^{\rho }V_{\gamma ,\beta ,\alpha }^{\delta ,p,q}\left( c\right) =0$, with
$c$ a constant, we  conclude that
\begin{equation*}
_{i}^{\rho }V_{\gamma ,\beta ,\alpha }^{\delta ,p,q}f\left( c\right) =\frac{%
f\left( b\right) -f\left( a\right) }{\displaystyle\frac{b^{\alpha }}{\alpha }%
-\frac{a^{\alpha }}{\alpha }}.
\end{equation*}
\end{proof}

The proof of the extension of the mean value theorem below will be omitted, since it develops is analogous as the $M$-fractional derivative \cite{JEC1,JEC}.

\begin{theorem}\label{teo7}{\rm\text{(Extension mean value theorem for fractional $\alpha$-differentiable functions)}} Let $a>0$ and $f,g:\left[ a,b\right] \rightarrow \mathbb{R}$ functions that satisfy:
\begin{enumerate}
\item $f, g$ are continuous in $[a,b]$.
\item $f, g$ are $\alpha$-differentiable for some $\alpha\in(0,1)$.
\end{enumerate}
Then, $\exists c\in(a,b)$, such that
\begin{equation*}
\frac{_{i}^{\rho }\mathcal{V}_{\gamma ,\beta ,\alpha }^{\delta ,p,q}f\left( c\right) }{_{i}^{\rho }\mathcal{V}_{\gamma ,\beta ,\alpha }^{\delta ,p,q}g\left( c\right) }=\frac{f\left( b\right) -f\left( a\right) }{g\left(b\right) -g\left( a\right) },
\end{equation*}
with $\gamma ,\beta ,\rho ,\delta \in \mathbb{C}$ and $p,q>0$ such that ${Re}\left( \gamma \right) >0$, ${Re}\left( \beta \right) >0$, ${Re}\left( \rho \right) >0$, ${Re}\left( \delta \right) >0$ and ${Re}\left( \gamma \right) +p\geq q$.
\end{theorem}

\begin{definition}\label{def8} Let $\alpha\in(n,n+1]$, for some $n\in\mathbb{N}$ and $f$ $n$-differentiable for $t>0$. Then the $n$-th derivative $\mathcal{V}$-fractional derivative of $f$ is defined by
\begin{equation}\label{PDS}
_{i}^{\rho }\mathcal{V}_{\gamma ,\beta ,\alpha }^{\delta ,p,q;n}f\left(
t\right) :=\underset{\varepsilon \rightarrow 0}{\lim }\frac{f^{(n)}\left(
t\;\Gamma \left( \beta \right)\; _{i}\mathbb{E}_{\gamma ,\beta ,p}^{\rho
,\delta ,q}\left( \varepsilon t^{n-\alpha }\right) \right) -f^{(n)}\left(
t\right) }{\varepsilon },
\end{equation}
with $\gamma ,\beta ,\rho ,\delta \in \mathbb{C}$ and $p,q>0$ such that ${Re}\left( \gamma \right) >0$, ${Re}\left( \beta \right) >0$, ${Re}\left( \rho \right) >0$, ${Re}\left( \delta \right) >0$ and ${Re}\left( \gamma \right) +p\geq q$, if the limit exists.
\end{definition}

From Definition \ref{def7} and the chain rule, that is, from item 5 of the Theorem \ref{teo2}, by induction on $n$, we can prove that $_{i}^{\rho }\mathcal{V}_{\gamma ,\beta ,\alpha }^{\delta ,p,q;n}f\left(t\right) =\displaystyle\frac{t^{n+1-\alpha }\Gamma \left( \beta \right) \left( \rho\right) _{q}}{\Gamma \left( \gamma +\beta \right) \left( \delta \right) _{p}}f^{\left( n+1\right) }\left( t\right) $, $\alpha\in(n,n+1]$ and $f$ is $(n+1)$-differentiable for $t>0$.

\section{$\mathcal{V}$-fractional integral}

In this section we introduce the $\mathcal{V}$-fractional integral of a function $f$. From the definition, we present a theorem showing that the $\mathcal{V}$-fractional integral is linear, the inverse property, the fundamental theorem of calculus, the part integration theorem, and a theorem that refer to the mean value for integrals. In addition, a theorem is given which returns the sum of the orders of two $\mathcal{V}$-fractional integrals, semi group property. Other results on the $\mathcal{V}$-fractional integral are also presented.

\begin{definition}\label{def9} {\rm($\mathcal{V}$-fractional integral)} Let $a\geq 0$ and $t\geq a$. Also, let $f$ be a function defined on $(a,t]$ and $0<\alpha<1$. Then, the $\mathcal{V}$-fractional integral of $f$ of order $\alpha$ is defined by
\begin{equation}\label{A15}
_{a}^{\rho }\mathcal{I}_{\gamma ,\beta ,\alpha }^{\delta ,p,q}f\left( t\right) :=\frac{\Gamma \left( \gamma +\beta \right) \left( \delta \right) _{p}}{\Gamma\left( \beta \right) \left( \rho \right) _{q}}\int_{a}^{t}\frac{f\left(x\right) }{x^{1-\alpha }}dx,
\end{equation}
with $\gamma ,\beta ,\rho ,\delta \in \mathbb{C}$ and $p,q>0$ such that ${Re}\left( \gamma \right) >0$, ${Re}\left( \beta \right) >0$, ${Re}\left( \rho \right) >0$, ${Re}\left( \delta \right) >0$ and ${Re}\left( \gamma \right) +p\geq q$.
\end{definition}

\begin{theorem}\label{st} Let $a\geq 0$ and $t\geq a$. Also, let $f,g:[a,t]\rightarrow\mathbb{R}$ continuous functions, such that exist $_{a}^{\rho }\mathcal{I}_{\gamma ,\beta ,\alpha }^{\delta ,p,q} f\left( t\right)$, $_{a}^{\rho }\mathcal{I}_{\gamma ,\beta ,\alpha }^{\delta ,p,q} g\left( t\right)$ with $0<\alpha<1$. Then, we have
\begin{enumerate}

\item $_{a}^{\rho }\mathcal{I}_{\gamma ,\beta ,\alpha }^{\delta ,p,q}\left( f\pm
g\right) \left( t\right) =\;_{a}^{\rho }\mathcal{I}_{\gamma ,\beta
,\alpha }^{\delta ,p,q}f\left( t\right) \pm \left( _{a}^{\rho }%
\mathcal{I}_{\gamma ,\beta ,\alpha }^{\delta ,p,q}g\left( t\right) \right) $.

\item $_{a}^{\rho }\mathcal{I}_{\gamma ,\beta ,\alpha }^{\delta ,p,q}\lambda
f\left( t\right) =\lambda \;_{a}^{\rho }\mathcal{I}_{\gamma ,\beta
,\alpha }^{\delta ,p,q}f\left( t\right) $.

\item If $t=a$, then $_{a}^{\rho }\mathcal{I}_{\gamma ,\beta ,\alpha }^{\delta ,p,q}f\left(
a\right) =0$.

\item If $f\left( x\right) \geq 0$, then $_{a}^{\rho }\mathcal{I}_{\gamma ,\beta ,\alpha }^{\delta ,p,q}f\left(t\right) \geq 0$.

\end{enumerate}
\end{theorem}

We will omit the proof of Theorem \ref{st}, since it follows from the Definition \ref{def9}.

\begin{theorem}\label{KJ} Let $a\geq 0$, $t\geq a$ and $f:[a,t]\rightarrow\mathbb{R}$ a continuous function, such that exist $_{a}^{\rho }\mathcal{I}_{\gamma ,\beta ,\alpha }^{\delta ,p,q} f\left( t\right)$, $_{a}^{\rho }\mathcal{I}_{\gamma ,\beta ,\mu }^{\delta ,p,q} f\left( t\right)$ with $0<\alpha<1$ and $0<\mu<1$. Then, we have
\begin{equation*}
_{a}^{\rho }\mathcal{I}_{\gamma ,\beta ,\alpha }^{\delta ,p,q}\left(
_{a}^{\rho }\mathcal{I}_{\gamma ,\beta ,\mu }^{\delta ,p,q}f\left( t\right)
\right) =\frac{\Gamma \left( \gamma +\beta \right) \left( \delta \right) _{p}
}{\Gamma \left( \beta \right) \left( \rho \right) _{q}}\left( \frac{
t^{\alpha }}{\alpha }\; _{a}^{\rho }\mathcal{I}_{\gamma ,\beta ,\alpha
}^{\delta ,p,q}f\left( t\right) -\frac{1}{\alpha }\;_{a}^{\rho }
\mathcal{I}_{\gamma ,\beta ,\alpha +\mu }^{\delta ,p,q}f\left( t\right) \right) .
\end{equation*}	
\end{theorem}

\begin{proof}In fact, using {\rm Definition \ref{def9}}, we have
\begin{eqnarray*}
_{a}^{\rho }\mathcal{I}_{\gamma ,\beta ,\alpha }^{\delta ,p,q}\left(
_{a}^{\rho }\mathcal{I}_{\gamma ,\beta ,\mu }^{\delta ,p,q}f\left( t\right)
\right)  &=&\frac{\Gamma \left( \gamma +\beta \right) \left( \delta \right)
_{p}}{\Gamma \left( \beta \right) \left( \rho \right) _{q}}%
\int_{a}^{t}\left( _{a}^{\rho }\mathcal{I}_{\gamma ,\beta ,\mu }^{\delta
,p,q}f\left( t\right) \right) x^{1-\alpha }dx  \notag \\
&=&\left( \frac{\Gamma \left( \gamma +\beta \right) \left( \delta \right)
_{p}}{\Gamma \left( \beta \right) \left( \rho \right) _{q}}\right)
^{2}\int_{a}^{t}\left( \int_{a}^{x}\frac{f\left( s\right) }{s^{1-\mu }}%
ds\right) x^{1-\alpha }dx  \notag \\
&=&\left( \frac{\Gamma \left( \gamma +\beta \right) \left( \delta \right)
_{p}}{\Gamma \left( \beta \right) \left( \rho \right) _{q}}\right)
^{2}\int_{a}^{t}\frac{f\left( s\right) }{s^{1-\mu }}\left( \frac{t^{\alpha }%
}{\alpha }-\frac{s^{\alpha }}{\alpha }\right) ds  \notag \\
&=&\frac{\Gamma \left( \gamma +\beta \right) \left( \delta \right) _{p}}{%
\Gamma \left( \beta \right) \left( \rho \right) _{q}}\left( \frac{t^{\alpha }%
}{\alpha }\; _{a}^{\rho }\mathcal{I}_{\gamma ,\beta ,\mu }^{\delta
,p,q}f\left( t\right)  -\frac{1}{\alpha }\ _{a}^{\rho }\mathcal{I%
}_{\gamma ,\beta ,\alpha +\mu }^{\delta ,p,q}f\left( t\right)\right). \notag \\
\end{eqnarray*}
\end{proof}

From  Theorem \ref{KJ}, we conclude that $_{a}^{\rho }\mathcal{I}_{\gamma ,\beta ,\alpha }^{\delta ,p,q}\left(_{a}^{\rho }\mathcal{I}_{\gamma ,\beta ,\mu }^{\delta ,p,q}f\left( t\right)
\right) \neq \left( _{a}^{\rho }\mathcal{I}_{\gamma ,\beta ,\alpha +\mu}^{\delta ,p,q}f\left( t\right) \right)$.

\begin{theorem}\label{teo8} {\rm(Reverse)} Let $a\geq 0$, $t\geq a$ and $0<\alpha<1$. Also, let $f$ be a continuous function such that $_{a}^{\rho }\mathcal{I}_{\gamma ,\beta ,\alpha }^{\delta ,p,q}f\left( t\right)$ exist. Then
\begin{equation*}
_{i}^{\rho }\mathcal{V}_{\gamma ,\beta ,\alpha }^{\delta ,p,q}\left(
_{a}^{\rho }\mathcal{I}_{\gamma ,\beta ,\alpha }^{\delta ,p,q}f\left(
t\right) \right) =f\left( t\right) ,
\end{equation*}
with $\gamma ,\beta ,\rho ,\delta \in \mathbb{C}$ and $p,q>0$ such that ${Re}\left( \gamma \right) >0$, ${Re}\left( \beta \right) >0$, ${Re}\left( \rho \right) >0$, ${Re}\left( \delta \right) >0$ and ${Re}\left( \gamma \right) +p\geq q$.
\end{theorem}
\begin{proof} In fact, using the chain rule and {\rm Eq.(\ref{def9})}, we have
\begin{eqnarray*}
_{i}^{\rho }\mathcal{V}_{\gamma ,\beta ,\alpha }^{\delta ,p,q}\left(
_{a}^{\rho }\mathcal{I}_{\gamma ,\beta ,\alpha }^{\delta ,p,q}f\left(
t\right) \right)  &=&\frac{t^{1-\alpha }\Gamma \left( \beta \right) \left( \rho \right) _{q}}{\Gamma \left( \gamma +\beta \right) \left( \delta \right) _{p}}\frac{d}{dt}\left( _{a}^{\rho }\mathcal{I}_{\gamma ,\beta ,\alpha }^{\delta ,p,q}f\left( t\right) \right)   \notag \\
&=&f\left( t\right).
\end{eqnarray*}
\end{proof}

\begin{theorem}\label{teo9} {\rm(Fundamental Theorem of Calculus)} Let $f:(a,b)\rightarrow\mathbb{R}$ be a differentiable function and $0<\alpha\leq 1$. Then, $\forall t>a$, we have
\begin{equation}\label{A16}
_{a}^{\rho }\mathcal{I}_{\gamma ,\beta ,\alpha }^{\delta ,p,q}\left(
_{i}^{\rho }\mathcal{V}_{\gamma ,\beta ,\alpha }^{\delta ,p,q}f\left(
t\right) \right) =f\left( t\right) -f\left( a\right),
\end{equation}
with $\gamma ,\beta ,\rho ,\delta \in \mathbb{C}$ and $p,q>0$ such that ${Re}\left( \gamma \right) >0$, ${Re}\left( \beta \right) >0$, ${Re}\left( \rho \right) >0$, ${Re}\left( \delta \right) >0$ and ${Re}\left( \gamma \right) +p\geq q$.
\end{theorem}
\begin{proof} In fact, since $f$ is a differentiable function, using the chain rule and the fundamental theorem of calculus for integer order derivatives, we have
\begin{eqnarray*}
_{a}^{\rho }\mathcal{I}_{\gamma ,\beta ,\alpha }^{\delta ,p,q}\left(
_{i}^{\rho }\mathcal{V}_{\gamma ,\beta ,\alpha }^{\delta ,p,q}f\left(
t\right) \right)  &=&\frac{\Gamma \left( \gamma +\beta \right) \left( \delta
\right) _{p}}{\Gamma \left( \beta \right) \left( \rho \right) _{q}}%
\int_{a}^{t}\frac{\left( _{i}^{\rho }\mathcal{V}_{\gamma ,\beta ,\alpha
}^{\delta ,p,q}\right) }{x^{1-\alpha }}f\left( x\right) dx  \notag \\
&=&f\left( t\right) -f\left( a\right).
\end{eqnarray*}
\end{proof}

If $f(a)=0$, then by Eq.(\ref{A16}), we have $_{a}^{\rho }\mathcal{I}_{\gamma ,\beta ,\alpha }^{\delta ,p,q}\left(_{i}^{\rho }\mathcal{V}_{\gamma ,\beta ,\alpha }^{\delta ,p,q}f\left(
t\right) \right) =f\left( t\right)$.

\begin{theorem}\label{teo10} Let $\gamma ,\beta ,\rho ,\delta \in \mathbb{C}$ and $p,q>0$ such that ${Re}\left( \gamma \right) >0$, ${Re}\left( \beta \right) >0$, ${Re}\left( \rho \right) >0$, ${Re}\left( \delta \right) >0$ and ${Re}\left( \gamma \right) +p\geq q$ and $f, g:[a,b]\rightarrow\mathbb{R}$ differentiable functions and $0<\alpha<1$. Then, we have
\begin{equation*}
\int_{a}^{b}f\left( x\right) \left( _{i}^{\rho }\mathcal{V}_{\gamma ,\beta
,\alpha }^{\delta ,p,q}g\left( x\right) \right) d_{\omega }x=f\left(
x\right) g\left( x\right) \mid _{a}^{b}-\int_{a}^{b}g\left( x\right) \left(
_{i}^{\rho }\mathcal{V}_{\gamma ,\beta ,\alpha }^{\delta ,p,q}f\left(
x\right) \right) d_{\omega }x,
\end{equation*}
with $d_{\omega }x=\displaystyle\frac{\Gamma \left( \gamma +\beta \right) \left( \delta
\right) _{p}}{\Gamma \left( \beta \right) \left( \rho \right) _{q}}\displaystyle\frac{dx}{
x^{1-\alpha }}$.
\end{theorem}
\begin{proof} 
In fact, by the definition of $\mathcal{V}$-fractional integral {\rm Eq.(\ref{A15})}, using the chain rule and the fundamental theorem of calculus for entire order derivatives, we have
\begin{eqnarray*}
\int_{a}^{b}f\left( x\right) \left( _{i}^{\rho }\mathcal{V}_{\gamma ,\beta
,\alpha }^{\delta ,p,q}g\left( x\right) \right) d_{\omega }x &=&\frac{\Gamma
\left( \gamma +\beta \right) \left( \delta \right) _{p}}{\Gamma \left( \beta
\right) \left( \rho \right) _{q}}\int_{a}^{b}f\left( x\right) \left(
_{i}^{\rho }\mathcal{V}_{\gamma ,\beta ,\alpha }^{\delta ,p,q}g\left(
x\right) \right) \frac{dx}{x^{1-\alpha }}  \notag \\
&=&\int_{a}^{b}f\left( x\right) g^{\prime }\left( x\right) dx  \notag \\
&=&f\left( x\right) g\left( x\right) \mid _{a}^{b}-\int_{a}^{b}g\left(
x\right) \left( _{i}^{\rho }\mathcal{V}_{\gamma ,\beta ,\alpha }^{\delta
,p,q}f\left( x\right) \right) d_{\omega }x.\notag \\
\end{eqnarray*}
\end{proof}

For the proof of {\rm Theorem \ref{teo11}} and {\rm Theorem \ref{teo12}}, simply use the definition of the integral $\mathcal{V}$-fractional {\rm Eq.(\ref{A15})} and follow the same as the respective theorems for the $M$-fractional integral, recently proposed by Sousa and Oliveira, \cite{JEC1,JEC}.

\begin{theorem}\label{teo11} Let $\gamma ,\beta ,\rho ,\delta \in \mathbb{C}$ and $p,q>0$ such that ${Re}\left( \gamma \right) >0$, ${Re}\left( \beta \right) >0$, ${Re}\left( \rho \right) >0$, ${Re}\left( \delta \right) >0$ and ${Re}\left( \gamma \right) +p\geq q$ and $f:[a,b]\rightarrow\mathbb{R}$ a continuous  function. Then, for $0<\alpha<1$, we have
\begin{equation}
\left\vert _{a}^{\rho }\mathcal{I}_{\gamma ,\beta ,\alpha }^{\delta
,p,q}f\left( t\right) \right\vert \leq (_{a}^{\rho }\mathcal{I}_{\gamma,\beta ,\alpha }^{\delta ,p,q}\left\vert f\left( t\right) \right\vert ).
\end{equation}
\end{theorem}

\begin{theorem}\label{teo12} Let $\gamma ,\beta ,\rho ,\delta \in \mathbb{C}$ and $p,q>0$ such that ${Re}\left( \gamma \right) >0$, ${Re}\left( \beta \right) >0$, ${Re}\left( \rho \right) >0$, ${Re}\left( \delta \right) >0$ and ${Re}\left( \gamma \right) +p\geq q$ and $f:[a,b]\rightarrow\mathbb{R}$ a continuous  function such that
\begin{equation*}
N=\underset{t\in \left[ a,b\right] }{\sup }\left\vert f\left( t\right)
\right\vert .
\end{equation*}

Then, $\forall t\in [a,b]$ and $0<\alpha<1$, we have
\begin{equation*}
\left\vert _{a}^{\rho }\mathcal{I}_{\gamma ,\beta ,\alpha }^{\delta
,p,q}f\left( t\right) \right\vert \leq \frac{\Gamma \left( \gamma +\beta \right) \left( \delta \right) _{p}}{\Gamma \left( \beta \right) \left( \rho \right) _{q}}N\left( \frac{t^{\alpha }}{\alpha }-\frac{a^{\alpha }}{\alpha } \right) .
\end{equation*}
\end{theorem}

\begin{theorem}\label{5a} Let $f$ and $g$ functions that satisfy the following conditions:
\begin{enumerate}
\item continuous in $[a,b]$,
\item limited and integrable in $[a,b]$.
\end{enumerate}
Besides that, let $g(x)$ a no negative (or no positive)  function in $[a,b]$. Let the set $m=\inf \left\{ f\left( x\right) :x\in \left[ a,b\right] \right\} $ and $M=\sup \left\{ f\left( x\right) :x\in \left[ a,b\right] \right\} $. Then, there exist a number $\xi\in(a,b)$ such that
\begin{equation}\label{bcn}
\int_{a}^{b}f\left( x\right) g\left( x\right) d_{\omega }x=\xi
\int_{a}^{b}g\left( x\right) d_{\omega }x,
\end{equation}
with $d_{\omega }x=\displaystyle\frac{\Gamma \left( \gamma +\beta \right) \left( \delta\right) _{p}}{\Gamma \left( \beta \right) \left( \rho \right) _{q}}\frac{dx}{x^{1-\alpha }}$.

If $f$ is continuous in $[a,b]$, then $\exists x_{0}\in [a,b]$, such that
\begin{equation*}
\int_{a}^{b}f\left( x\right) g\left( x\right) d_{\omega }x=f\left(
x_{0}\right) \int_{a}^{b}g\left( x\right) d_{\omega }x.
\end{equation*}
\end{theorem}

\begin{proof} Let $m=\inf f$, $M=\sup f$ and $g(x)\geq 0$ in $[a,b]$. Then, we have
\begin{equation}\label{klm}
mg\left( x\right) <f\left( x\right) g\left( x\right) <Mg\left( x\right) .
\end{equation}

Multiplying by $\displaystyle\frac{\Gamma \left( \gamma +\beta \right) \left( \delta \right) _{p}}{\Gamma \left( \beta \right) \left( \rho \right) _{q}x^{1-\alpha }}$ on both sides of {\rm Eq.(\ref{klm})} and integrating with respect to the variable $x$ on $(a,b)$, we get
\begin{equation}\label{zezinho}
m\int_{a}^{b}\frac{\Gamma \left( \gamma +\beta \right) \left( \delta \right)
_{p}}{\Gamma \left( \beta \right) \left( \rho \right) _{q}}\frac{g\left(
x\right) }{x^{1-\alpha }}dx<\int_{a}^{b}\frac{\Gamma \left( \gamma +\beta
\right) \left( \delta \right) _{p}}{\Gamma \left( \beta \right) \left( \rho
\right) _{q}}\frac{f\left( x\right) g\left( x\right) }{x^{1-\alpha }}%
dx<M\int_{a}^{b}\frac{\Gamma \left( \gamma +\beta \right) \left( \delta
\right) _{p}}{\Gamma \left( \beta \right) \left( \rho \right) _{q}}\frac{%
g\left( x\right) }{x^{1-\alpha }}dx.
\end{equation}

Then, $\exists\xi\in[m,M]$, such that
\begin{equation*}
\int_{a}^{b}f\left( x\right) g\left( x\right) d_{\omega }x=\xi
\int_{a}^{b}g\left( x\right) d_{\omega }x,
\end{equation*}
with $d_{\omega }x=\displaystyle\frac{\Gamma \left( \gamma +\beta \right) \left( \delta
\right) _{p}}{\Gamma \left( \beta \right) \left( \rho \right) _{q}}\frac{dx}{
x^{1-\alpha }}$.

Note that, when $g(x)<0$, the proof is performed in an analogous way.

Moreover, by the intermediate value theorem, $f$ reaches each value in the interval $[m,M]$, then to $x_{0}$ in $[a,b]$, $f(x_{0})=\xi$. So, we have

\begin{equation*}
\int_{a}^{b}f\left( x\right) g\left( x\right) d_{\omega }x=f\left(
x_{0}\right) \int_{a}^{b}g\left( x\right) d_{\omega }x.
\end{equation*}

If $g(x)=0$, {\rm Eq.(\ref{bcn})} becomes obvious and if $g(x)>0$, then, {\rm Eq.(\ref{zezinho})} implies
\begin{equation}
m<\frac{\int_{a}^{b}f\left( x\right) g\left( x\right) d_{\omega }x}{%
\int_{a}^{b}g\left( x\right) d_{\omega }x}<M,
\end{equation}
exist a point $x_{0}\in(a,b)$ such that $m<f\left( x_{0}\right) <M$, the result follows.

In particular, when $g(x)=1$, by {\rm Theorem \ref{5a}}, we have the result
\begin{eqnarray*}
\int_{a}^{b}f\left( x\right) d_{\omega }x &=&f\left( x_{0}\right) \frac{%
\Gamma \left( \gamma +\beta \right) \left( \delta \right) _{p}}{\Gamma
\left( \beta \right) \left( \rho \right) _{q}}\int_{a}^{b}\frac{1}{%
x^{1-\alpha }}dx  \notag \\
&=&f\left( x_{0}\right) \frac{\Gamma \left( \gamma +\beta \right) \left(
\delta \right) _{p}}{\Gamma \left( \beta \right) \left( \rho \right) _{q}}%
\left( \frac{b^{\alpha }}{\alpha }-\frac{a^{\alpha }}{\alpha }\right),
\end{eqnarray*}
or in the following form
\begin{eqnarray}\label{vc}
f\left( x_{0}\right)  &=&\frac{\Gamma \left( \beta \right) \left( \rho
\right) _{q}}{\Gamma \left( \gamma +\beta \right) \left( \delta \right) _{p}}%
\frac{1}{\left( \frac{b^{\alpha }}{\alpha }-\frac{a^{\alpha }}{\alpha }%
\right) }\int_{a}^{b}f\left( x\right) d_{\omega }x  \notag \\
&=&\frac{1}{\left( \frac{b^{\alpha }}{\alpha }-\frac{a^{\alpha }}{\alpha }%
\right) }\int_{a}^{b}\frac{f\left( x\right) }{x^{1-\alpha }}dx.
\end{eqnarray}

{\rm Eq.(\ref{vc})} is the so-called average value of the $f$ function. When $\alpha=1$, we have the mean value theorem for integrals
\begin{equation*}
\int_{a}^{b}f\left( x\right) dx=\left( b-a\right) f\left( x_{0}\right).
\end{equation*}
\end{proof}

To conclude this section, we will present below some relations of the $\mathcal{V}$-fractional integral {\rm Eq.(\ref{A15})} with other fractional integrals.

Choosing $\beta=\rho=\delta=p=q=1$ and substituting in {\rm Eq.(\ref{A15})}, we have
$_{a}^{1}\mathcal{I}_{\gamma ,1,\alpha }^{1,1,1}f\left( t\right) =\Gamma \left( \gamma
+1\right) \displaystyle\int_{a}^{t}\frac{f\left( x\right) }{x^{1-\alpha }}dx=\;_{M}\mathcal{I}_{a}^{\gamma ,\alpha }f\left( t\right) $, which is exactly the $M$-fractional integral, recently introduced by Sousa and Capelas \cite{JEC1,JEC}.

On the order hand, for $\gamma=1$, we conclude $_{a}^{1}\mathcal{I}_{1,1,\alpha }^{1,1,1}f\left( t\right) =\displaystyle\int_{a}^{t}\frac{f\left(x\right) }{x^{1-\alpha }}dx=I_{a}^{\alpha }f\left( t\right) $, exactly the $\alpha$-fractional integral recently introduced by Khalil et al. \cite{KRHA} and Katugampola \cite{UNT2}.
\section{Derivative and integral $\mathcal{V}$-fractional of a Mittag-Leffler function}

Mittag-Leffler functions are important in the theory of fractional calculus and are present in the solutions of fractional differential equations. There are several applications with such functions, from fractional derivative definitions to Laplace transform and derivative calculus. Next, we calculate the $\mathcal{V}$-fractional derivative of two parameters Mittag-Leffler function, {\rm Eq.(\ref{def2})}.

\begin{theorem} Let $_{i}^{\rho }\mathcal{V}_{\gamma ,\beta ,\alpha }^{\delta ,p,q}f(t)$ the truncated $\mathcal{V}$-fractional derivative of order $\alpha$ and $\mathbb{E}_{\mu,\kappa}(\cdot)$ the two parameters Mittag-Leffler function {\rm Eq.(\ref{def2})}. Then, we have
\begin{equation*}
_{i}^{\rho }\mathcal{V}_{\gamma ,\beta ,\alpha}^{\delta ,p,q}\left( \mathbb{E}_{\mu
,\kappa }\left( t\right) \right) =t^{1-\alpha }\frac{\Gamma \left( \beta
\right) \left( \rho \right) _{q}}{\Gamma \left( \gamma +\beta \right) \left(
\delta \right) _{p}}\mathbb{E}_{\mu ,\mu +\kappa }^{2}\left( t\right) ,
\end{equation*}
where $\mathbb{E}_{\mu,\kappa}^{\rho}(\cdot)$ is the three parameters Mittag-Leffler function.
\end{theorem}
\begin{proof}
In fact, using the chain rule and the two parameters Mittag-Leffler function, we have

\begin{eqnarray}\label{PO1}
_{i}^{\rho }\mathcal{V}_{\gamma ,\beta ,\alpha }^{\delta ,p,q}\left( \mathbb{E}_{\mu
,\kappa }\left( t\right) \right)  &=&t^{1-\alpha }\frac{\Gamma \left( \beta
\right) \left( \rho \right) _{q}}{\Gamma \left( \gamma +\beta \right) \left(
\delta \right) _{p}}\frac{d}{dt}\left( \overset{\infty }{\underset{k=0}{\sum 
}}\frac{t^{k}}{\Gamma \left( \mu k+\kappa \right) }\right)   \notag \\
&=&t^{1-\alpha }\frac{\Gamma \left( \beta \right) \left( \rho \right) _{q}}{%
\Gamma \left( \gamma +\beta \right) \left( \delta \right) _{p}}\overset{%
\infty }{\underset{k=0}{\sum }}\frac{k\left( k-1\right) !}{\left( k-1\right)
!}\frac{t^{k-1}}{\Gamma \left( \mu k+\kappa \right) }.
\end{eqnarray}

Exchanging the index $k \rightarrow k+1$, in {\rm Eq.(\ref{PO1})}, we have
\begin{eqnarray*}
_{i}^{\rho }\mathcal{V}_{\gamma ,\beta ,\alpha}^{\delta ,p,q}\left( \mathbb{E}_{\mu
,\kappa }\left( t\right) \right)  &=&t^{1-\alpha }\frac{\Gamma \left( \beta
\right) \left( \rho \right) _{q}}{\Gamma \left( \gamma +\beta \right) \left(
\delta \right) _{p}}\overset{\infty }{\underset{k=0}{\sum }}\frac{\left(
k+1\right) !}{k!}\frac{t^{k}}{\Gamma \left( \mu k+\mu +\kappa \right) } 
\notag \\
&=&t^{1-\alpha }\frac{\Gamma \left( \beta \right) \left( \rho \right) _{q}}{%
\Gamma \left( \gamma +\beta \right) \left( \delta \right) _{p}}\mathbb{E}_{\mu ,\mu
+\kappa }^{2}\left( t\right).
\end{eqnarray*}
\end{proof}

\begin{theorem} Let $_{i}^{\rho }\mathcal{V}_{\gamma ,\beta ,\alpha }^{\delta ,p,q;n}f(t)$ the truncated $\mathcal{V}$-fractional derivative of order $\alpha$ and $\mathbb{E}_{\mu,\kappa}(\cdot)$ the two parameters Mittag-Leffler function. Then, we have
\begin{equation}
_{i}^{\rho }\mathcal{V}_{\gamma ,\beta ,\alpha }^{\delta ,p,q;n}\left( \mathbb{E}_{\mu ,\kappa }\left( t\right) \right) =t^{1-\alpha }\frac{\Gamma\left( \beta \right) \left( \rho \right) _{q}}{\Gamma \left( \gamma +\beta\right) \left( \delta \right) _{p}}\Gamma \left( n+2\right) \mathbb{E}_{\mu,\kappa +\mu \left( n+1\right) }^{n+2}\left( t\right) ,
\end{equation}
with $n=0,1,2,...$.
\end{theorem}

\begin{proof}
Consider the following result {\rm \cite{TGS}}
\begin{equation}\label{H12}
\frac{d^{n}}{dt^{n}}\left( \mathbb{E}_{\mu ,\kappa }^{\rho ,q}\left(
t\right) \right) =\left( \rho \right) _{qn}\mathbb{E}_{\mu ,\kappa +\mu n}^{\rho +qn,q}\left( t\right) ,
\end{equation}
where $\mathbb{E}^{\rho,q}_{\mu,\kappa}(\cdot)$ is the four parameters Mittag-Leffler function.

In particular, for $\rho =q=1$ in {\rm Eq.(\ref{H12})}, we have
\begin{equation}\label{H13}
\frac{d^{n}}{dt^{n}}\left( \mathbb{E}_{\mu ,\kappa }\left( t\right) \right)=\Gamma \left( n+1\right) \mathbb{E}_{\mu ,\kappa +\mu n}^{n+1}\left(t\right) ,
\end{equation}
where $\mathbb{E}^{\rho}_{\mu,\kappa}(\cdot)$ is the three parameters Mittag-Leffler function.

Taking the entire order derivative on both sides of the {\rm Eq.(\ref{H13})} and choosing $q=n=1$ in {\rm Eq.(\ref{H12})},
\begin{equation*}
\frac{d}{dt}\left( \mathbb{E}_{\mu ,\kappa +\mu n}^{n+1}\left( t\right)\right) =\frac{\Gamma \left( n+2\right) }{\Gamma \left( n+1\right) }\mathbb{E}_{\mu ,\kappa +\mu \left( n+1\right) }^{n+2}\left( t\right),
\end{equation*}
we get
\begin{equation}\label{H14}
\frac{d^{n+1}}{dt^{n+1}}\left( \mathbb{E}_{\mu ,\kappa }\left( t\right) \right) =\Gamma \left( n+2\right) \mathbb{E}_{\mu ,\kappa +\mu \left( n+1\right) }^{n+2}\left( t\right) .
\end{equation}

Using the chain rule and {\rm Eq.(\ref{H14})}, we conclude
\begin{eqnarray*}
_{i}^{\rho }\mathcal{V}_{\gamma ,\beta ,\alpha }^{\delta ,p,q;n}\left( \mathbb{E}_{\mu ,\kappa }\left( t\right) \right)  &=&t^{n+1-\alpha }\frac{\Gamma \left( \beta \right) \left( \rho \right) _{q}}{\Gamma \left( \gamma +\beta \right) \left( \delta \right) _{p}}\frac{d^{n+1}}{dt^{n+1}}\left(  \mathbb{E}_{\mu ,\kappa }\left( t\right) \right)   \notag \\ 
&=&t^{n+1-\alpha }\frac{\Gamma \left( \beta \right) \left( \rho \right) _{q}}{\Gamma \left( \gamma +\beta \right) \left( \delta \right) _{p}}\Gamma \left( n+2\right) \mathbb{E}_{\mu ,\kappa +\mu \left( n+1\right) }^{n+2}\left( t\right) .
\end{eqnarray*}
\end{proof}

Let us now calculate the $\mathcal{V}$-fractional integrate of the two parameters Mittag-Leffler function.

\begin{theorem} Let $_{a}^{\rho }\mathcal{I}_{\gamma ,\beta ,\alpha }^{\delta ,p,q}f(t)$ the $\mathcal{V}$-fractional integral of order $\alpha$ and $\mathbb{E}_{\mu,\kappa}(\cdot)$ the two parameters Mittag-Leffler function. Then, we have
\begin{equation*}
_{a}^{\rho }\mathcal{I}_{\gamma ,\beta ,\alpha }^{\delta ,p,q}\left( \mathbb{E}_{\mu
,\kappa }\left( t\right) \right) =\frac{\Gamma \left( \gamma +\beta \right)
\left( \delta \right) _{p}}{\Gamma \left( \beta \right) \left( \rho \right)
_{q}}\left( t^{\alpha }\mathbb{E}_{\mu +1,\kappa +\alpha +1}\left( t\right)
-a^{\alpha }\mathbb{E}_{\mu +1,\kappa +\alpha +1}\left( a\right) \right) .
\end{equation*}
\end{theorem}

\begin{proof} In fact, using the definition of $\mathcal{V}$-fractional integral {\rm Eq.(\ref{A15})} and the fundamental theorem of calculus, we have
\begin{eqnarray}\label{josy}
_{a}^{\rho }\mathcal{I}_{\gamma ,\beta ,\alpha }^{\delta ,p,q}\left( \mathbb{%
E}_{\mu ,\kappa }\left( t\right) \right)  &=&\frac{\Gamma \left( \gamma
+\beta \right) \left( \delta \right) _{p}}{\Gamma \left( \beta \right)
\left( \rho \right) _{q}}\int_{a}^{t}x^{\alpha -1}\underset{k=0}{\overset{%
\infty }{\sum }}\frac{x^{k}}{\Gamma \left( \mu k+\kappa \right) }dx \notag \\
&=&\frac{\Gamma \left( \gamma +\beta \right) \left( \delta \right) _{p}}{%
\Gamma \left( \beta \right) \left( \rho \right) _{q}}\underset{k=0}{\overset{%
\infty }{\sum }}\frac{1}{\Gamma \left( \mu k+\kappa \right) }%
\int_{a}^{t}x^{\alpha +k-1}dx \notag \\
&=&\frac{\Gamma \left( \gamma +\beta \right) \left( \delta \right) _{p}}{%
\Gamma \left( \beta \right) \left( \rho \right) _{q}}\underset{k=0}{\overset{%
\infty }{\sum }}\frac{1}{\Gamma \left( \mu k+\kappa \right) }\left( \frac{%
t^{k+\alpha }}{k+\alpha }-\frac{a^{k+\alpha }}{k+\alpha }\right) \notag \\
&=&\frac{\Gamma \left( \gamma +\beta \right) \left( \delta \right) _{p}}{%
\Gamma \left( \beta \right) \left( \rho \right) _{q}}\left( 
\begin{array}{c}
t^{\alpha }\underset{k=0}{\overset{\infty }{\sum }}\frac{t^{k}}{\Gamma
\left( \left( \mu +1\right) k+\kappa +\alpha +1\right) } \\ 
-a^{\alpha }\underset{k=0}{\overset{\infty }{\sum }}\frac{a^{k}}{\Gamma
\left( \left( \mu +1\right) k+\kappa +\alpha +1\right) }%
\end{array}%
\right)  \\
&=&\frac{\Gamma \left( \gamma +\beta \right) \left( \delta \right) _{p}}{%
\Gamma \left( \beta \right) \left( \rho \right) _{q}}\left( t^{\alpha }%
\mathbb{E}_{\mu +1,\kappa +\alpha +1}\left( t\right) -a^{\alpha }\mathbb{E}%
_{\mu +1,\kappa +\alpha +1}\left( a\right) \right). \notag
\end{eqnarray}
\end{proof}

In particular, taking the limit $a \rightarrow 0$, on both sides of Eq.(\ref{josy}), we have that the unique contributing factor is $k=0$,
\begin{equation}\label{josy1}
\underset{a\rightarrow 0}{\lim }\left( a^{\alpha }\underset{k=0}{\overset{\infty }{\sum }}\frac{a^{k}}{\Gamma \left( \left( \mu +1\right) k+\kappa
+\alpha +1\right) }\right) =\underset{a\rightarrow 0}{\lim }a^{\alpha }\frac{1}{\Gamma \left( \kappa +\alpha +1\right) }=0.
\end{equation}

In this sense, from Eq.(\ref{josy}) and Eq.(\ref{josy1}), we conclude that
\begin{equation*}
_{0}^{\rho }\mathcal{I}_{\gamma ,\beta ,\alpha }^{\delta ,p,q}\left( \mathbb{%
E}_{\mu ,\kappa }\left( t\right) \right) =\frac{\Gamma \left( \gamma +\beta
\right) \left( \delta \right) _{p}}{\Gamma \left( \beta \right) \left( \rho
\right) _{q}}t^{\alpha }\mathbb{E}_{\mu +1,\kappa +\alpha +1}\left( t\right) .
\end{equation*}

\section{Relation with the derivative and integral of Riemann-Liouville }

In this section we present the demonstration of two theorems that relate the $\mathcal{V}$-fractional derivative with the fractional derivative in the Riemann-Liouville sense and the $\mathcal {V}$-fractional integral with Riemann-Liouville fractional integral.

\begin{theorem}\label{teo20} Let $_{0}^{\rho }\mathcal{I}_{\gamma ,\beta ,\alpha }^{\delta ,p,q}f(t)$ the $\mathcal{V}$-fractional integral of order $\alpha$, $0<\alpha<1$ with $a=0$ and the function $f(t)=(t-x)^{\mu}$, $t>x$ and $\mu>-1$. Then, we have
\begin{equation*}
_{0}^{\rho }\mathcal{I}_{\gamma ,\beta ,\alpha }^{\delta ,p,q}\left(
t-x\right) ^{\mu }=\frac{\Gamma \left( \gamma +\beta \right) \Gamma \left(
\alpha \right) \left( \delta \right) _{p}}{\Gamma \left( \beta \right)
\left( \rho \right) _{q}} J^{\alpha }t^{\mu }, 
\end{equation*}
where $J^{\alpha }t^{\mu } $ is the Riemann-Liouville fractional integral of order $\alpha$ {\rm \cite{AHMJ,IP}}.
\end{theorem}

\begin{proof} In fact, by definition of the $\mathcal{V}$-fractional integral, we have
\begin{eqnarray}\label{T1}
_{0}^{\rho }\mathcal{I}_{\gamma ,\beta ,\alpha }^{\delta ,p,q}\left(
t-x\right) ^{\mu } &=&\frac{\Gamma \left( \gamma +\beta \right) \left(
\delta \right) _{p}}{\Gamma \left( \beta \right) \left( \rho \right) _{q}}%
\int_{0}^{t}t^{\mu }\left( 1-\frac{x}{t}\right) ^{\mu }x^{\alpha -1}dx 
\notag \\
&=&\frac{\Gamma \left( \gamma +\beta \right) \left( \delta \right) _{p}}{%
\Gamma \left( \beta \right) \left( \rho \right) _{q}}t^{\mu +\alpha
}\int_{0}^{1}\left( 1-u\right) ^{\mu }u^{\alpha -1}du  \notag \\
&=&\frac{\Gamma \left( \gamma +\beta \right) \left( \delta \right) _{p}}{%
\Gamma \left( \beta \right) \left( \rho \right) _{q}}t^{\mu +\alpha }\frac{%
\Gamma \left( \mu +1\right) \Gamma \left( \alpha \right) }{\Gamma \left( \mu
+1+\alpha \right) }.
\end{eqnarray}

Consider the following result {\rm\cite{RM}},
\begin{equation}\label{T2}
J^{\alpha }t^{\mu }=\frac{\Gamma \left( \mu +1\right) }{\Gamma \left( \mu
+1+\alpha \right) }t^{\mu +\alpha },\text{ }t>0\text{ }and\text{ }\mu >-1.
\end{equation}

Thus, from {\rm Eq.(\ref{T1})} and {\rm Eq.(\ref{T2})}, we conclude that
\begin{equation*}
_{0}^{\rho }\mathcal{I}_{\gamma ,\beta ,\alpha }^{\delta ,p,q}\left(
t-x\right) ^{\mu }=\frac{\Gamma \left( \gamma +\beta \right) \Gamma \left(
\alpha \right) \left( \delta \right) _{p}}{\Gamma \left( \beta \right)
\left( \rho \right) _{q}}J^{\alpha }t^{\mu },
\end{equation*}
where $J^{\alpha }(\cdot)$ is the Riemann-Liouville fractional integral of order $\alpha$.
\end{proof}

\begin{theorem} Let $_{0}^{\rho }\mathcal{I}_{\gamma ,\beta ,\alpha }^{\delta ,p,q}f(t)$ and $_{i}^{\rho }\mathcal{V}_{\gamma ,\beta ,\alpha }^{\delta ,p,q;n}f(t)$ the $\mathcal{V}$-fractional integral and derivative of order $\alpha$, $0<\alpha<1$ with $a=0$ and the function $f(t)=(t-x)^{\mu}$, $t>x$ and $\mu>-1$. Then, we have
\begin{equation*}
_{i}^{\rho }\mathcal{V}_{\gamma ,\beta ,\alpha }^{\delta ,p,q}\left(
_{0}^{\rho }\mathcal{I}_{\gamma ,\beta ,\alpha }^{\delta ,p,q}\left(
t-x\right) ^{\mu }\right) =\frac{\Gamma \left( \mu +1-\alpha \right) \Gamma
\left( \alpha \right) }{\Gamma \left( \mu +1+\alpha \right) }t^{\alpha
}\left( \mu +\alpha \right) \mathcal{D}^{\alpha }t^{\mu },
\end{equation*}
where $\mathcal{D}^{\alpha }t^{\mu }$ is the Riemann-Liouville fractional derivative of order $\alpha$ {\rm \cite{AHMJ,IP}}.
\end{theorem}
\begin{proof} Using the chain rule, and {\rm Theorem \ref {teo20}}, we have
\begin{eqnarray}\label{T3}
_{i}^{\rho }\mathcal{V}_{\gamma ,\beta ,\alpha }^{\delta ,p,q}\left(
_{0}^{\rho }\mathcal{I}_{\gamma ,\beta ,\alpha }^{\delta ,p,q}\left(
t-x\right) ^{\mu }\right)  &=&\frac{t^{1-\alpha }\Gamma \left( \beta \right)
\left( \rho \right) _{q}}{\Gamma \left( \gamma +\beta \right) \left( \delta
\right) _{p}}\frac{d}{dt}\left( _{0}^{\rho }\mathcal{I}_{\gamma ,\beta
,\alpha }^{\delta ,p,q}\left( t-x\right) ^{\mu }\right)   \notag \\
&=&t^{1-\alpha }\Gamma \left( \beta \right) \frac{d}{dt}\left( \frac{\Gamma
\left( \mu +1\right) }{\Gamma \left( \mu +1+\alpha \right) }t^{\mu +\alpha
}\right)   \notag \\
&=&t^{\alpha }\frac{\Gamma \left( \alpha \right) \Gamma \left( \mu +1\right) 
}{\Gamma \left( \mu +1+\alpha \right) }\left( \mu +\alpha \right) t^{\mu
-\alpha }.
\end{eqnarray}

Consider the following result {\rm \cite{RM}},
\begin{equation}\label{T4}
\mathcal{D}^{\alpha }t^{\mu }=\frac{\Gamma \left( \mu +1\right) }{\Gamma \left( \mu
+1-\alpha \right) }t^{\mu -\alpha },\text{ }t>0\text{ }e\text{ }\mu >-1.
\end{equation}

Thus, from {\rm Eq.(\ref{T3})} and {\rm Eq.(\ref{T4})}, we conclude that
\begin{equation*}
_{i}^{\rho }\mathcal{V}_{\gamma ,\beta ,\alpha }^{\delta ,p,q}\left(
_{0}^{\rho }\mathcal{I}_{\gamma ,\beta ,\alpha }^{\delta ,p,q}\left(
t-x\right) ^{\mu }\right) =\frac{\Gamma \left( \alpha \right) \Gamma \left(
\mu +1-\alpha \right) }{\Gamma \left( \mu +1+\alpha \right) }t^{\alpha
}\left( \mu +\alpha \right) \mathcal{D}^{\alpha }t^{\mu },
\end{equation*}
where $\mathcal{D}^{\alpha}t^{\mu}$ is the Riemann-Liouville fractional derivative.
\end{proof}


\section{Relation with other fractional derivatives}

In this section, we will discuss the relationship between the truncated $\mathcal{V}$-fractional derivative and the conformable fractional derivative, alternative fractional derivative and truncated alternative fractional derivative, as well as the $M$-fractional derivative and the truncated $M$-fractional derivative.

Recently Sousa and Oliveira \cite{JEC}, proposed the truncated $M$-fractional derivative given by
\begin{equation}\label{T5}
_{i}\mathscr{D}_{M}^{\gamma,\alpha }f\left( t\right) =\underset{\varepsilon
\rightarrow 0}{\lim }\frac{f\left( t\;_{i}\mathbb{E}_{\gamma }\left( \varepsilon t^{-\alpha }\right) \right) -f\left( t\right) }{\varepsilon },
\end{equation}
$\forall t>0$ and $_{i}\mathbb{E}_{\gamma }(\cdot)$, $\gamma>0$ is the one parameter truncated Mittag-Leffler function.

Here we consider the truncated function $_{i}H^{\rho,\delta,q}_{\gamma,\beta,p}(\cdot)$, which is defined in terms of the product of a gamma function and the six parameters truncated Mittag-Leffler function Eq.(\ref{A9}). Then, choosing $p=q=\delta=\rho=\beta=1$ in the definition of the truncated function $_{i}H^{\rho,\delta,q}_{\gamma,\beta,p}(\cdot)$, we have
\begin{equation}\label{T6}
_{i}H_{\gamma ,1,1}^{1,1,1}\left( \varepsilon t^{-\alpha }\right) =\overset{i%
}{\underset{k=0}{\sum }}\frac{\left( \varepsilon t^{-\alpha }\right) ^{k}}{%
\Gamma \left( \gamma k+1\right) }=\;_{i}\mathbb{E}_{\gamma }\left( \varepsilon
t^{-\alpha }\right) .
\end{equation}

Thus, from Eq.(\ref{A11}) and {\rm Eq.(\ref{T6})}, we conclude
\begin{eqnarray}
_{i}^{1}\mathcal{V}_{\gamma ,1,\alpha }^{1,1,1}f\left( t\right)  &=&\underset%
{\varepsilon \rightarrow 0}{\lim }\frac{f\left( t_{i}H_{\gamma
,1,1}^{1,1,1}\left( \varepsilon t^{-\alpha }\right) \right) -f\left(
t\right) }{\varepsilon }  \notag \\
&=&\underset{\varepsilon \rightarrow 0}{\lim }\frac{f\left( t\;_{i}\mathbb{E}%
_{\gamma }\left( \varepsilon t^{-\alpha }\right) \right) -f\left( t\right) }{%
\varepsilon }  \notag \\
&=&\;_{i}\mathscr{D}_{M}^{\gamma ,\alpha }f\left( t\right) ,
\end{eqnarray}
which is exactly the truncated $M$-fractional derivative {\rm Eq.(\ref{T5})}. Consequently, from the $M$-fractional derivative, we obtain a relation for other four existing fractional derivatives: conformable, alternative, truncated alternative, and $M$-fractional \cite{KRHA,UNT2,JEC1,JEC}.

The diagram below contains particular cases obtained from the truncated $\mathcal{V}$-fractional derivative, the conformable fractional derivative, the alternative fractional derivative, the truncated alternative fractional derivative, and the $M$-fractional derivative, besides to truncated $M$-fractional derivative.
\vspace*{1cm}

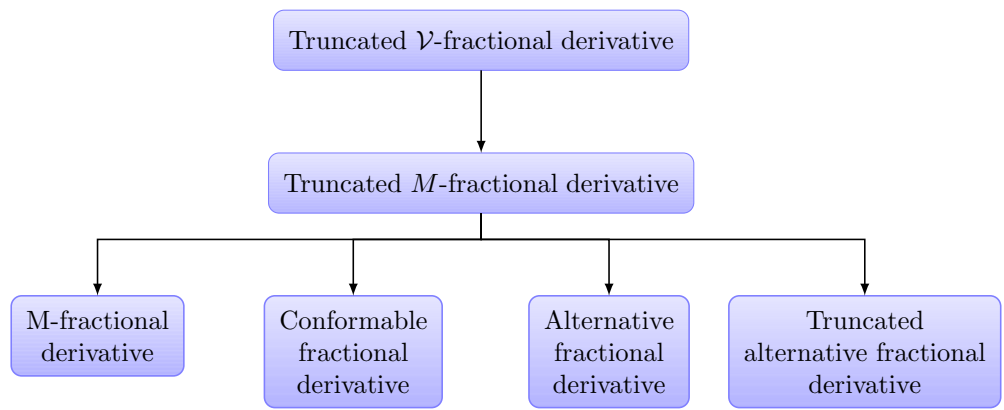
\begin{footnotesize}
\centering
\tikzstyle{block} = [rectangle, draw=blue, thick, fill=blue!20, align=center,
text width=16.5em, node distance=5.0cm, text centered, rounded corners,  minimum width=5cm, minimum height=5em]

\tikzstyle{block2} = 
    [rectangle, draw=blue, thick, fill=blue!20, align=center,
text width=8em, node distance=5.0cm, text centered, rounded corners,  minimum width=5cm, minimum height=5em]

\tikzstyle{block3} = 
    [rectangle, draw=blue, thick, fill=blue!20, align=center,
text width=8em, node distance=5.0cm, text centered, rounded corners,  minimum width=5cm, minimum height=5em]

\tikzstyle{line} = 
    [draw, -latex']

\begin{tikzpicture}[
        node distance=0mm,
    every node/.style = {shape=rectangle, rounded corners, draw=blue!50,
                   inner sep=2mm, outer sep=0mm, minimum height=8mm,
                   align=center, anchor=north},
    cmnt/.style = {shape=rectangle, draw=none,
                   inner sep=0mm, outer sep=0mm, minimum height=8mm,
                   align=right},
             level 1/.style = {sibling distance=42mm},
             level 2/.style = {sibling distance=34mm},
    edge from parent fork down,
    edge from parent/.style = {draw, semithick, -latex},                            ]
\begin{scope}[every node/.append style={top color=blue!10, bottom color=blue!30}]
    \node   (L1)   {Truncated $\mathcal{V}$-fractional derivative}
        child {node (L2) {Truncated $M$-fractional derivative}
            child {node (L3)  {M-fractional\\ derivative}}
            child {node {Conformable \\fractional\\ derivative}}
            child {node {Alternative\\ fractional\\ derivative}}
            child {node {Truncated \\alternative fractional\\ derivative}}
                }
                  ;
\end{scope}

        \end{tikzpicture}

\end{footnotesize}


\section{Concluding remarks}

We have introduced a new truncated $\mathcal{V}$-fractional derivative for $\alpha$-differentiable functions using the six parameters truncated Mittag-Leffler function and the gamma function and the corresponding $\mathcal {V}$-fractional integral.

For a class of $\alpha$-differentiable functions, in the context of fractional derivatives, we conclude that the truncated $\mathcal{V}$-fractional derivative, proposed here, behaves very well in relation to the classical properties of entire order calculus. In addition, it was possible through of truncated $\mathcal{V}$-fractional derivative and $\mathcal{V}$-fractional integral, to obtain relations with the fractional derivative and fractional integral in the Riemann-Liouville sense. In this sense, with our fractional derivative, it was possible to obtain a very important relation with the fractional derivatives mentioned, as seen in section 7.

We conclude that the results presented here contain as particular cases the derivatives proposed by Khalil et al. \cite{KRHA}, Katugampola \cite{UNT2} and Sousa and Oliveira \cite{JEC1,JEC}. As future works, we can introduce the truncated $k$-Mittag-Leffler function from the $k$-Mittag-Leffler function \cite{DGCRA,NKEA} and propose a generalization of fractional derivatives. Also, the presented results are related to the function of a variable, in this sense, we can propose a truncated $\mathcal{V}$-fractional derivative with $n$ real variables \cite{GNY}. Studies in this direction will be published \cite{JVE,JVE1}.

\bibliography{ref}
\bibliographystyle{plain}

\end{document}